\documentclass[a4paper,11pt]{article}
\usepackage{mathrsfs}
\usepackage{latexsym}
\usepackage{amsmath,amssymb}
\usepackage[pdftex]{hyperref}
\usepackage{amsthm}
\usepackage{amsfonts,cite}
\usepackage[usenames]{color}
\usepackage{amssymb}
\usepackage{graphicx}
\usepackage{amsmath}
\usepackage{amsfonts}
\usepackage{amsthm}
\usepackage{mathrsfs}
\usepackage{dsfont}
\usepackage{indentfirst}

\newtheoremstyle{mythm}{1.5ex plus 1ex minus .2ex}{1.5ex plus 1ex
minus .2ex}{\kai}{\parindent}{\song\bfseries}{}{1em}{}
\numberwithin{equation}{section}

\newcommand{\be}{\begin{equation}}
\newcommand{\ee}{\end{equation}}

\newcommand{\ptl}{\partial}

\newcommand{\eps}{\varepsilon}

\newcommand{\lam}{\lambda}
\numberwithin{equation}{section}

\newtheorem{lemma}{Lemma}[section]
\newtheorem{rem}{Remark}
\newtheorem{thm}{Theorem}[section]
\newtheorem{pro}{Proposition}[section]

\allowdisplaybreaks[4]
\begin{document}
\title{{\textbf{Classification of solutions of higher order critical Choquard equation}}}
\author{Genggeng Huang\footnote{genggenghuang@fudan.edu.cn} and Yating Niu\footnote{ytniu19@fudan.edu.cn}}
\date{}
\maketitle
\begin{center}
School of Mathematical Sciences, Fudan University, Shanghai, China
\end{center}

\begin{abstract}
In this paper, we classify the solutions of the following critical Choquard equation
\[
(-\Delta)^{\frac{n}{2}} u(x) = \int_{\mathbb{R}^n} \frac{e^{\frac{2n- \mu}{2}u(y)}}{|x-y|^{\mu}}dy e^{\frac{2n- \mu}{2}u(x)}, \ \text{in} \ \mathbb{R}^n,
\]
where $ 0<\mu < n$, $ n\ge 2$. Suppose $ u(x) = o(|x|^2) \ \text{at} \ \infty $ for $ n \geq 3$ and satisfies
\begin{equation}
\int_{\mathbb{R}^n}e^{\frac{2n- \mu}{2}u(y)} dy < \infty, \ \int_{\mathbb{R}^n}\int_{\mathbb{R}^n}\frac{e^{\frac{2n- \mu}{2}u(y)}}{|x-y|^{\mu}} e^{\frac{2n- \mu}{2}u(x)} dy dx < \infty. \tag{B}
\end{equation}
By using the method of moving spheres, we show that the solutions have the following form
\[
u(x)= \ln  \frac{C_1(\eps)}{|x-x_0|^2 + \eps^2}.
\]
\end{abstract}
\section{Introduction}
In this paper, we classify the solutions of the following Choquard equation
\begin{equation}\label{c1}
(-\Delta)^{\frac{n}{2}} u(x) = \int_{\mathbb{R}^n} \frac{e^{\frac{2n- \mu}{2}u(y)}}{|x-y|^{\mu}}dy e^{\frac{2n- \mu}{2}u(x)}, \ \text{in} \ \mathbb{R}^n,
\end{equation}
where $ 0<\mu < n $ and $ n\ge 2$. When $ n$ is odd, the nonlocal fractional Laplacians $ (-\Delta)^{\frac{n}{2}}$ are defined by
\[
(-\Delta)^{\frac{n}{2}} u = (-\Delta)^{\frac{1}{2}} \circ (-\Delta)^{\frac{n-1}{2}}u(x).
\]
Interested readers may refer to the book of Chen et al. \cite{ChenLiMa}. In order to insure the problem is well-defined, we assume throughout this paper that for a $ C^n$ solution $ u$ of \eqref{c1}, the following assumption
\begin{equation}\label{c2}
\int_{\mathbb{R}^n}e^{\frac{2n- \mu}{2}u(y)} dy < \infty, \ \int_{\mathbb{R}^n}\int_{\mathbb{R}^n}\frac{e^{\frac{2n- \mu}{2}u(y)}}{|x-y|^{\mu}} e^{\frac{2n- \mu}{2}u(x)} dy dx < \infty \tag{B}
\end{equation}
holds true.

Consider the conformally flat manifolds $(M,g)=(\mathbb R^n, e^{2u}|dx|^2)$, where $n\ge 2$ is an even integer. Then the $Q-$curvature $Q_g$ satisfies
\begin{equation}\label{0905}
(-\Delta)^{\frac n2} u=Q_g e^{nu},\quad \text{in}\quad \mathbb R^n.
\end{equation}
For more details of $Q-$curvatures, we refer readers to \cite{Branson1985,GJMS1992}. For \eqref{c1}, we can view
$Q_g=e^{-\frac{\mu}2 u(x)}\int_{\mathbb R^n}\frac{e^{\frac{2n-\mu}{2}u(y)}}{|x-y|^\mu}dy$ as the corresponding $Q-$curvature.
Since the volume form $dvol_g=e^{nu} dx$,
the first condition in  \eqref{c2} is the finite volume condition with weight $e^{-\frac{\mu}{2}u}$ and the second condition in  \eqref{c2} is the finite total curvature condition.

The Choquard equation
\begin{equation}\label{c9}
-\Delta u + V(x)u =2 \left( \frac{1}{|x|} \ast |u|^{2}\right) u, \ x \in \mathbb{R}^n
\end{equation}
has various applications in mathematical physics, such as quantum mechanics (see for \cite{Pekar}). Mathematically, Lieb \cite{Lieb} proved the existence and uniqueness of the minimal energy solution of \eqref{c9} by using rearrangement technique, when $ V$ is a positive constant. For $ V\equiv 1$, Ma and Zhao \cite{MaZhao} proved the positive solutions of \eqref{c9} must be radially symmetric and monotone decreasing about some fixed point. For more results on Choquard equations, see for \cite{GaoYang,Lions,MVVJ}.

For the critical Choquard equation
\[
-\Delta u = \left( \frac{1}{|x|^\mu} \ast u^{2_\mu^*}\right) u^{2_\mu^*-1}, \ x \in \mathbb{R}^n, \ n\ge 3,
\]
Du and Yang \cite{DuYang19} applied the moving plane method to classify the positive solution of the nonlocal equation, where $ 2_\mu^* = \frac{2n-\mu}{n-2}$, $ 0< \mu < n$, if $ n=3$ or $ 4$ and $ 0< \mu \le 4$, if $ n\ge 5$. $ 2_\mu^*$ is the upper critical exponent in the sense of the Hardy-Littlewood-Sobolev inequality. Guo et al. \cite{GuoHuPeng} proved the same result and they also studied the nonlinear Choquard equation. For multi-critical Choquard equations, please refer to \cite{LiuYangYu}.

In \cite{Yu}, Yang and Yu consider the following Choquard equation in dimension two
\begin{equation}\label{c5}
-\Delta u(x) = \int_{\mathbb{R}^2} \frac{e^{\frac{4-\mu}{2}u(y)}}{|x-y|^{\mu}} dy e^{\frac{4-\mu}{2}u(x)} \ \text{in} \ \mathbb{R}^2,
\end{equation}
where $ 0< \mu <1$. They classified the solution of the above equation \eqref{c5} under the assumptions
\[
\int_{\mathbb{R}^2} e^{\frac{4-\mu}{2}u(y)} dy < \infty, \quad \int_{\mathbb{R}^2}\int_{\mathbb{R}^2} \frac{e^{\frac{4-\mu}{2}u(y)}e^{\frac{4-\mu}{2}u(x)}}{|x-y|^{\mu}} dydx< \infty.
\]

The natural generalizations of \eqref{c5} is the higher order critical Choquard equation \eqref{c1}. We rewrite \eqref{c1} into the following differential-integral system
\begin{equation}\label{c6}
\left\{
\begin{aligned}
& (-\Delta)^{\frac{n}{2}} u(x) = v(x) e^{\frac{2n- \mu}{2}u(x)} \quad \text{in} \quad  \mathbb{R}^n, \\
& v(x) = \int_{\mathbb{R}^n} \frac{e^{\frac{2n- \mu}{2}u(y)}}{|x-y|^{\mu}}dy \quad  \text{in} \quad  \mathbb{R}^n. 
\nonumber
\end{aligned}
\right.
\end{equation}

Now let us consider the following integral equation
\begin{equation}\label{c7}
u(x) = \frac{1}{\beta_n}  \int_{\mathbb{R}^n}\left[ \ln \left(\frac{|y|+1}{|x-y|}\right)\right] v(y) e^{\frac{2n-\mu}{2}u(y)} dy + C,
\end{equation}
where
\[ v(x) = \int_{\mathbb{R}^n} \frac{e^{\frac{2n- \mu}{2}u(y)}}{|x-y|^{\mu}}dy \quad \text{and} \quad \beta_n = 2^{n-1} \left(\frac{n}{2}\right)! \left(\frac{n}{2} -1\right)! \omega_n.
\]
$ \omega_n$ is the volume of the unit ball in $ \mathbb{R}^n$. We establish the equivalence between differential equation \eqref{c1} and integral equation \eqref{c7}.
\begin{thm}\label{c24}
Suppose $ u \in C^n$ is a solution of \eqref{c1} which satisfies \eqref{c2} and
\begin{equation}\label{c3}
u(x) = o(|x|^2) \ \text{at} \ \infty \quad \text{for} \quad n\geq 3.
\end{equation}
Then $ u$ also satisfies \eqref{c7}.
\end{thm}

\begin{thm}\label{c61}
Suppose $ u \in C^n$ is a solution of \eqref{c1} which satisfies \eqref{c2} and \eqref{c3}. Then we have
\[
u(x) = -\alpha \ln |x| + O(1)
\]
for $ |x|$ large enough, where
\[
\alpha := \frac{1}{\beta_n} \int_{\mathbb{R}^n} v(y) e^{\frac{2n- \mu}{2}u(y)} dy >\frac{2n}{2n-\mu}
\]
and
\[
\lim_{|x|\rightarrow \infty} |x|^{\mu} v(x) = \beta,
\]
where
\[
\beta = \int_{\mathbb{R}^n}e^{\frac{2n- \mu}{2}u(y)} dy.
\]
\end{thm}

With the aids of the equivalence of differential equation \eqref{c1} and integral equation \eqref{c7} and the asymptotic behavior of $ (u,v)$, we deduce the Liouville type theorem for solutions $ u$ to \eqref{c1}.

\begin{thm}\label{c4}
Assume that $ u \in C^n$ is a solution to \eqref{c1} and satisfies the hypothesis \eqref{c2} and \eqref{c3}. Then for $ n\geq 2$ and $ \mu \in(0,n)$, $ u$ has the following form
\[
u(x)= \ln  \frac{C_1(\eps)}{|x-x_0|^2 + \eps^2},
\]
where $ C_1$ is a positive constant depending only on $ \eps$ and $ x_0$ is a fixed point in $ \mathbb{R}^n$.
\end{thm}

\begin{rem}
For the case $ n=2$ and $ 0< \mu<1$, Yang and Yu \cite{Yu} proved Theorem \ref{c4} and gave the decay of $ u$ at $ \infty$ without the condition \eqref{c3}. Actually, the same results are also true for $ 0< \mu <2$.
\end{rem}

Theorem \ref{c4} is proved by the moving spheres method, which is a variant of the moving planes method. It is a very powerful tool to study the symmetry of solutions. Many classification results were obtained, see for instance \cite{ChenLi91,ChenLi09,ChenLiOu,HuangNiu,Yu22}.

Theorem \ref{c4} generalizes the result of \cite{Yu} to higher order equation. Firstly, in \cite{Yu}, they used the moving spheres method to rule out both slow decay and fast decay, then obtained the exact decay of the solution $ u$. In this paper, inspired by \cite{Xu05}, we use the Pohozaev identity and give the decay of the solution $ u$ directly. We show preliminarily that
\[
\lim_{|x| \rightarrow \infty} \frac{u(x)}{ \ln |x|} = -\alpha \quad \text{and} \quad \alpha >\frac{2n}{2n-\mu}.
\]
However, the proof of the exact value $ \alpha =2$ is more complicated. Since the corresponding $ v(x)$ in \cite{Xu05} is a constant. We have made some improvements to their methods. Secondly, Yang and Yu \cite{Yu} classified the solutions of \eqref{c5} in dimension two for $ 0<\mu <1$. In this paper, we generalize the dimension from $ 2$ to $ n$ and the index $ \mu$ to $ 0<\mu <n$. The main difficulty is the integral definition of $ \nabla v(x)$ is not well-defined for $ n-1\le \mu <n$. We use techniques such as integration by parts, cut off function and integral estimation to deal with this problem, see Lemma \ref{c65}.

If $ \mu =0$, the Choquard equation \eqref{c1} formally  reduces to
\begin{equation}\label{c59}
\left\{
\begin{aligned}
& (-\Delta)^{\frac{n}{2}} u =(n-1)!e^{nu} \quad \text{in} \quad \mathbb{R}^n,\\
& \int_{\mathbb{R}^n} e^{nu(x)}dx < \infty \quad \text{and} \quad u=o(|x|^2) \ \text{at} \ \infty.
\end{aligned}
\right.
\end{equation}
Zhu \cite{Zhu} obtained the Liouville theorems to \eqref{c59} with $ n=3$. For $ n=4$, Lin \cite{Lin98} classified the fourth order equation of \eqref{c59}. In the case that $ n$ is an even integer, Wei and Xu \cite{XuWei} worked out the explicit form of the solution from \eqref{c59}. For all dimensions $ n \ge 3$, Xu \cite{Xu05} gave the equivalent integral form and the classification of the solution by moving spheres method.  For more literatures on higher order equations, please refer to \cite{DaiLiuQin,HuangLi,LiM1}.

This paper is organized as follows. In Section 2, we give the integral representation formula for $ u$. And we obtain the asymptotic behavior of $ u$ and $ v$ at $ \infty$. In Section 3, we use the moving spheres method to prove Theorem \ref{c4}.
\section{Preliminaries}
In this section, we first establish the equivalence between the differential equation \eqref{c1} and integral equation \eqref{c7}. Furthermore, we study the decay of solution $ u$ at $ \infty$.

\begin{lemma}\label{c25}
Suppose $ u \in C^2(\mathbb{R}^n)$ such that \\
$(a)$
\[
\alpha := \frac{1}{\beta_n} \int_{\mathbb{R}^n} v(y) e^{\frac{2n- \mu}{2}u(y)} dy < \infty \qquad \int_{\mathbb{R}^n}e^{\frac{2n- \mu}{2}u(y)} dy < \infty,
\]
where
\[
v(x) = \int_{\mathbb{R}^n} \frac{e^{\frac{2n- \mu}{2}u(y)}}{|x-y|^{\mu}}dy;
\]
$ (b)$ $ u$ satisfies the following equation:
\[
\Delta u(x) + v(x) e^{\frac{4- \mu}{2}u(x)} = 0 \quad \text{for} \quad n=2,
\]
\begin{equation}\label{c23}
\Delta u(x) + (n-2) \int_{\mathbb{R}^n} \frac{v(y) e^{\frac{2n- \mu}{2}u(y)}}{|x-y|^2} dy =0 \quad \text{for} \quad  n\geq 3.
\end{equation}
Then there is a constant $ C > 0$ such that
\begin{equation}\label{c8}
u(x) \le C \quad \text{and} \quad v(x)\le C \quad \text{in} \quad \mathbb{R}^n.
\end{equation}
\end{lemma}

Before we prove the above Lemma \ref{c25}, we need to recall two results.
\begin{lemma}\label{c26} (see Lemma 3.2 of \cite{Xu05})
Suppose $ u \in C^2(\mathbb{R}^n)$ such that $ 0\leq -\Delta u(x) \le C$ in $ \mathbb{R}^n$ and $ \int_{\mathbb{R}^n}e^{\frac{2n- \mu}{2}u(y)} dy < \infty$. Then there exists a constant $ C >0$, such that
\[
 u(x) \le C \quad \text{in} \quad   \mathbb{R}^n.
\]
\end{lemma}

\begin{lemma}\label{c36} (see Theorem 1 of \cite{BreMer})
Let $ h(x)$ be a solution of
\begin{equation}
\left\{
\begin{aligned}
-\Delta h(x) & = f(x) \quad  \text{in} \quad B_R \subset \mathbb{R}^2 \\
h(x) & =0 \qquad \text{on} \quad \partial B_R
\end{aligned}
\right.\nonumber
\end{equation}
with $ f\in L^1(B_R)$, then for any $ \delta \in (0,4\pi)$, there exists a constant $ C_{\delta}>0$ such that
\[
\int_{B_R} \exp\left( \frac{(4\pi -\delta)|h(x)| }{\|f\|_{L^1(B_R)}}\right) dx \leq \frac{4\pi^2}{\delta} R^2.
\]
\end{lemma}

\begin{proof}[Proof of Lemma \ref{c25}] Firstly, we consider $ n\geq 3$. For any $ 0< \eps < \min \left\{ \frac{2(n-\mu)}{(2n-\mu)n}, \frac{2(n-2)}{(2n-\mu)n}\right\}$, there exists an $ R>0$ sufficiently large such that
\be\label{c22}
\int_{\mathbb{R}^n\setminus B_R} v(y) e^{\frac{2n- \mu}{2}u(y)} dy \leq \eps.
\ee
Now $ \forall x_0 \in \mathbb{R}^n\setminus B_{R+8}$, consider the solution $ h$ of the equation
\begin{equation}\label{c21}
\left\{
\begin{aligned}
-\Delta h(x) & = (n-2) \int_{B_4(x_0)} \frac{v(y) e^{\frac{2n- \mu}{2}u(y)}}{|x-y|^2} dy \quad \text{in} \ B_4(x_0) \\
h& =0 \qquad \qquad \qquad \qquad \qquad \qquad \qquad \text{on} \ \ptl B_4(x_0).
\end{aligned}
\right.
\end{equation}
From the maximum principle we conclude that
\[
h(x) \ge 0, \quad x \in B_4(x_0).
\]
And also consider the function
\[
h_1(x) = \int_{B_4(x_0)} \ln \left(\frac{16}{|x-y|}\right) v(y) e^{\frac{2n- \mu}{2}u(y)} dy , \quad \forall x\in B_4(x_0).
\]
We conclude that
\[
h_1(x) \geq 0 \quad \text{in} \quad B_4(x_0).
\]
By a simple calculation, we have
\be\label{c20}
-\Delta h_1(x) = (n-2) \int_{B_4(x_0)} \frac{v(y) e^{\frac{2n- \mu}{2}u(y)}}{|x-y|^2} dy \quad \text{in} \quad B_4(x_0).
\ee
Combining \eqref{c21} and \eqref{c20} yields that
\begin{equation}
\left\{
\begin{aligned}
-\Delta ( h -h_1) = 0 \quad \text{in} \ B_4(x_0)\\
h -h_1 \le 0 \qquad \text{on} \ \ptl B_4(x_0).
\end{aligned}
\right. \nonumber
\end{equation}
The maximum principle allows us to conclude that
\[
h(x) \le h_1(x),  \quad x\in B_4(x_0).
\]

Now let us denote the measure $ v(y) e^{\frac{2n- \mu}{2}u(y)} dy / \int_{B_4(x_0)}v(y) e^{\frac{2n- \mu}{2}u(y)} dy $ by $ d\mu$. Therefore Jenson's inequality, together with \eqref{c22}, imply that
\begin{equation}
\begin{aligned}
\int_{B_4(x_0)} e^{\frac{h(x)}{\eps}} dx & \le \int_{B_4(x_0)} \exp\left(\frac{h_1(x)}{\int_{B_4(x_0)}v(y) e^{\frac{2n- \mu}{2}u(y)} dy}\right) dx \\
& = \int_{B_4(x_0)} \exp\left( \int_{B_4(x_0)} \ln \left(\frac{16}{|x-y|}\right) d \mu\right) dx \\
& \le \int_{B_4(x_0)} \int_{B_4(x_0)} \frac{16}{|x-y|} d \mu dx \\
& = \int_{B_4(x_0)} \int_{B_4(x_0)} \frac{16}{|x-y|} dx d \mu \\
& \le C. \nonumber
\end{aligned}
\end{equation}
We consider the function
\[
q(x) = u(x) - h(x), \quad x\in B_3(x_0).
\]
We find that
\begin{equation}
\begin{aligned}
\Delta q(x) &= \Delta u(x) - \Delta h(x) = - (n-2) \int_{\mathbb{R}^n \setminus B_4(x_0)} \frac{v(y) e^{\frac{2n- \mu}{2}u(y)}}{|x-y|^2} dy.
\nonumber
\end{aligned}
\end{equation}
If $ x\in B_3(x_0)$ and $ y \in \mathbb{R}^n \setminus B_4(x_0)$, then $ |x-y|\geq 1$. Therefore, one has
\[
0\le -\Delta q(x) \le (n-2)\beta_n \alpha.
\]
Hence it follows from Harnack principle (Theorem 8.17 of \cite{GT}) that
\[
\sup_{B_2(x_0)} q(x) \le C(\|q^+\|_{L^2(B_3(x_0))} + \|\Delta q\|_{L^\infty(B_3(x_0))}).
\]
To estimate the first term, we note that
\[
q^+(x) \le u^+(x).
\]
Thus we get
\begin{equation}
\begin{aligned}
\int_{B_3(x_0)} (q^+(x))^2 dx & \le C \int_{B_3(x_0)} e^{\frac{2n-\mu}{2} q^+(x)} dx  \\
& \le  C \int_{B_3(x_0)} e^{\frac{2n-\mu}{2} u^+(x)} dx \\
& \le C.
\nonumber
\end{aligned}
\end{equation}
Therefore, it follows that
\[
u(x) = q(x) + h(x) \le C +h(x), \quad x \in B_2(x_0).
\]
Therefore we reach the estimate
\[
\int_{B_2(x_0)} e^{\frac{u(y)}{\eps}} dy  \le C\int_{B_2(x_0)} e^{\frac{ h(y)}{\eps}} dy \le \tilde{C}.
\]
By the definition of $ v(x)$, we get
\begin{equation}
\begin{aligned}
v(x) &= \int_{\mathbb{R}^n} \frac{e^{\frac{2n-\mu}{2}u(y)}}{|x-y|^{\mu}}dy \\
& = \int_{\mathbb{R}^n \setminus B_1(x)} \frac{e^{\frac{2n-\mu}{2}u(y)}}{|x-y|^{\mu}}dy + \int_{B_1(x)} \frac{e^{\frac{2n-\mu}{2}u(y)}}{|x-y|^{\mu}}dy \\
& \le C + \left(\int_{B_1(x)} \frac{1}{|x-y|^{p\mu}}dy \right)^{\frac{1}{p}} \left(\int_{B_1(x)} e^{\frac{2n-\mu}{2}qu(y)} dy \right)^{\frac{1}{q}} \\
& \le C. \nonumber
\end{aligned}
\end{equation}
In getting last inequality of above, we need $ 1< p < \frac{n}{\mu}$, $ \frac{1}{p} +\frac{1}{q} =1$ and $ q> \frac{n}{n-\mu}$. According to $ \frac{1}{\eps} > \frac{n(2n-\mu)}{2(n-\mu)}$, we can choose $ q $ satisfying
\[
\frac{1}{\eps} \ge \frac{2n-\mu}{2} q > \frac{n(2n-\mu)}{2(n-\mu)},
\]
such that $ v(x) \le C$. By \eqref{c23}, we have for any $ |x_0|$ sufficiently large,
\begin{equation}
\begin{aligned}
-\Delta u(x_0) &= (n-2) \int_{\mathbb{R}^n} \frac{v(y) e^{\frac{2n- \mu}{2}u(y)}}{|x_0-y|^2} dy \\
& = (n-2) \int_{\mathbb{R}^n \setminus B_2(x_0)} \frac{v(y) e^{\frac{2n- \mu}{2}u(y)}}{|x_0-y|^2} dy + (n-2) \int_{B_2(x_0)} \frac{v(y) e^{\frac{2n- \mu}{2}u(y)}}{|x_0-y|^2} dy \\
& \le C+C \left(\int_{B_2(x_0)} \frac{1}{|x_0-y|^{2p_1}} dy \right)^{\frac{1}{p_1}}  \left(\int_{B_2(x_0)} e^{\frac{2n- \mu}{2}q_1 u(y)} dy \right)^{\frac{1}{q_1}} \\
& \le C,\nonumber
\end{aligned}
\end{equation}
where $ \frac{1}{p_1} +\frac{1}{q_1} =1$, $1< p_1 < \frac{n}{2}$ and $ q_1 > \frac{n}{n-2}$. Since $ \frac{1}{\eps} > \frac{n(2n-\mu)}{2(n-2)}$, we can choose $ q_1$ satisfying
\[
\frac{1}{\eps}\ge \frac{2n- \mu}{2}q_1 > \frac{n(2n-\mu)}{2(n-2)}
\]
such that $ -\Delta u(x_0) \le C$.
Therefore, $ -\Delta u \geq 0$ is bounded on $ \mathbb{R}^n$. According to Lemma \ref{c26}, we get directly
\[
u(x) \leq C,
\]
for some constant $ C>0$.

Then for the case $ n=2$, by the regularity theory of elliptic equations, we see that
\[
\sup_{B_1(x_0)} u \leq C \{\|u^+\|_{L^2(B_2(x_0))} + \|ve^{\frac{4-\mu}{2}u}\|_{L^2(B_2(x_0))}\}
\]
Similarly, we consider that $ h(x)$ and $ q(x)$ satisfy the following equations
\begin{equation}
\left\{
\begin{aligned}
-\Delta h(x) & = v(x)e^{\frac{4-\mu}{2}u(x)} \quad \text{in} \quad B_4(x_0),\\
h(x) & =0 \qquad \text{on} \quad \partial B_4(x_0)
\end{aligned}
\right.
\end{equation}
and
\begin{equation}
\left\{
\begin{aligned}
-\Delta g(x) & =0 \qquad \text{in} \quad B_4(x_0),\\
g(x) & =u(x) \quad \text{on} \quad \partial B_4(x_0).
\end{aligned}
\right.
\end{equation}
We know that
\[
u(x) = h(x) + g(x)\quad \text{in} \quad B_4(x_0).
\]
Using Lemma \ref{c36} and the mean value property of harmonic function, we can prove that
\[
u(x) \leq C \quad \text{and} \quad v(x) \leq C
\]
for some constant $ C>0$. The proof is similar to the case for $ n\geq 3$ and is omitted.
\end{proof}

Define
\begin{equation}\label{c56}
J(x) =\frac{1}{\beta_n}  \int_{\mathbb{R}^n}\left[ \ln \left(\frac{|y|+1}{|x-y|}\right)\right] v(y) e^{\frac{2n-\mu}{2}u(y)} dy.
\end{equation}
Since $ \int_{\mathbb{R}^n}v(y) e^{\frac{2n-\mu}{2}u(y)} dy < \infty$, $ J(x)$ is well-defined.

\begin{lemma}\label{c68}
Suppose that $ u$ satisfies condition \eqref{c2}. Then there holds
\[
\lim_{|x| \rightarrow \infty} \frac{J(x)}{ \ln |x|} = -\alpha.
\]
\end{lemma}
\begin{proof}
Firstly, we prove that
\begin{equation}\label{c13}
J(x) \geq -\alpha \ln|x|,
\end{equation}
for $ |x|$ large. Set
\[
D_1=\left\{y \in \mathbb{R}^n | |y-x|\leq \frac{|x|}{2} \right\}\ \text{and} \ \ D_2=\left\{y \in \mathbb{R}^n | |y-x| > \frac{|x|}{2}\right\}.
\]
For $ y\in D_1$, we have $ |x-y|\leq \frac{|x|}{2} \leq |y| < |y|+1$, hence
\[
\ln \frac{|y|+1}{|x-y|} > 0,
\]
which further implies
\[
J(x) \geq \frac{1}{\beta_n} \int_{D_2} \left[ \ln \left(\frac{|y|+1}{|x-y|}\right)\right]  v(y) e^{\frac{2n-\mu}{2} u(y)} dy.
\]
For $ |x|\geq 2$, we find $ |x-y| \leq |x|+|y| \leq |x|(1+|y|)$ and
\[
\ln \frac{|y|+1}{|x-y|} \geq \ln \frac{1}{|x|}.
\]
For $ |x|\geq 2$, we conclude that
\begin{align*}
J(x)& \geq \frac{-\ln|x|}{\beta_n}  \int_{D_2}  v(y) e^{\frac{2n-\mu}{2} u(y)}  dy \\
& \geq -\alpha \ln|x|.
\end{align*}

Then, we claim that $ \forall \eps >0$, there exists an $ R_\eps > 0$ such that
\[
J(x) \le -(\alpha - \eps )\ln|x|, \quad \forall |x| \ge R_\eps.
\]
Let $ A_1=\{y \in \mathbb{R}^n||y|\leq R_0\}$. Then we can choose $ R_0$ large enough, such that
\begin{equation}\label{b47}
\frac{1}{\beta_n} \int_{A_1} [\ln|x-y|- \ln(|y|+1)]v(y) e^{\frac{2n-\mu}{2}u(y)} dy \geq \left(\alpha - \frac{\eps}{2}\right) \ln|x|.
\end{equation}
Let $ A_2=\left\{y \in \mathbb{R}^n ||y-x|\leq \frac{|x|}{2}, |y|\geq R_0 \right\}$ and $ A_3=\left\{y \in \mathbb{R}^n||y-x|> \frac{|x|}{2}, |y|\geq R_0 \right\}$. Then
\begin{equation}\label{b72}
\begin{aligned}
& \frac{1}{\beta_n}\int_{A_2} [\ln|x-y|-\ln(|y|+1)]v(y) e^{\frac{2n-\mu}{2}u(y)} dy \\
\geq & \frac{1}{\beta_n}\int_{B_1(x)}\ln|x-y|v(y) e^{\frac{2n-\mu}{2}u(y)}  dy - \frac{1}{\beta_n}\int_{A_2} \ln(|y|+1)v(y) e^{\frac{2n-\mu}{2}u(y)}  dy \\
\geq & -C - \frac{\ln(2|x|)}{\beta_n} \int_{A_2} v(y) e^{\frac{2n-\mu}{2}u(y)}  dy \\
\geq & -C - \frac{\eps}{4} \ln|x|.
\end{aligned}
\end{equation}

If $ y\in A_3$, then we find $ |y-x| >\frac{|x|}{2} \geq \frac{1}{2}(|y|-|x-y|)$, i.e.,
\[
\frac{|x-y|}{|y|} \geq \frac{1}{3}
\]
or
\[
\frac{|x-y|}{|y|+1} \geq \frac{1}{6}.
\]
Hence, it is clear that
\begin{equation}\label{b73}
\begin{aligned}
\frac{1}{\beta_n} \int_{A_3}[\ln|x-y|-\ln(|y|+1)]v(y)e^{\frac{2n-\mu}{2}u(y)}  dy & \geq \frac{-\ln 6}{\beta_n}  \int_{A_3} v(y)e^{\frac{2n-\mu}{2}u(y)}dy \\
& \geq - \frac{\eps}{4} \ln|x|
\end{aligned}
\end{equation}
for $ |x|$ enough large.
Finally, we infer from \eqref{b47}-\eqref{b73} that
\[
-J(x)\geq (\alpha -\eps)\ln|x|
\]
for $ |x|$ large enough. This proves the claim.
\end{proof}

\begin{thm}\label{c28}
Let $ u \in C^{n}$ be the solution of the following equation
\begin{equation}
\left\{
\begin{aligned}
& (-\Delta)^{\frac{n}{2}} u(x) = v(x) e^{\frac{2n- \mu}{2}u(x)} \quad \text{in} \quad  \mathbb{R}^n, \\
& v(x) = \int_{\mathbb{R}^n} \frac{e^{\frac{2n- \mu}{2}u(y)}}{|x-y|^{\mu}}dy \quad  \text{in} \quad  \mathbb{R}^n, \\
\end{aligned}
\right. \nonumber
\end{equation}
where $ n\ge 2$. If $ u(x) = o(|x|^2)$ for $ n\geq 3$ and
\be\label{c10}
\alpha := \frac{1}{\beta_n} \int_{\mathbb{R}^n} v(y) e^{\frac{2n- \mu}{2}u(y)} dy < \infty \qquad \int_{\mathbb{R}^n}e^{\frac{2n- \mu}{2}u(y)} dy < \infty,
\ee
then $ u$ is given by
\be\label{c29}
u(x) = \frac{1}{\beta_n}  \int_{\mathbb{R}^n}\left[ \ln \left(\frac{|y|+1}{|x-y|}\right)\right] v(y) e^{\frac{2n-\mu}{2}u(y)} dy + C,
\ee
for some constant $ C$.
\end{thm}
\begin{proof} By \eqref{c56}, $ J(x)$ satisfies
\[
(-\Delta)^{\frac{n}{2}} J(x) = v(x) e^{\frac{2n- \mu}{2}u(x)}.
\]
Set $ p= [\frac{n+1}{2}]$, the greatest integer less than or equal to $ \frac{n+1}{2}$. We can see that the function $ u - J$ is a poly-harmonic function with $ (-\Delta)^p(u-J) = 0$.

Claim: If $ p\ge 2$, $ (-\Delta)^{p-1}(u-J) = 0 $.

Let $ g = u-J$. Since $ (\Delta)^{p-1}g $ is harmonic, by mean value theorem and divergence theorem, we have $ \forall x_0 \in \mathbb{R}^n$ and $ \forall r>0$,
\begin{equation}\label{c11}
\begin{aligned}
&[(\Delta)^{p-1}g](x_0) \\
= &\frac{1}{\omega_n r^n} \int_{B_r(x_0)} [(\Delta)^{p-1}g] (y)dy \\
=& \frac{1}{\omega_n r^n} \int_{\partial B_r(x_0)} \frac{\ptl}{\ptl r}[(\Delta)^{p-2}g] (y)dS \\
=& \frac{n}{r}\frac{1}{n\omega_n r^{n-1}} \int_{\partial B_r(x_0)} \frac{\ptl}{\ptl r}[(\Delta)^{p-2}g] (y)dS \\
= &\frac{n}{r}\frac{\ptl}{\ptl r} \left(\frac{1}{n\omega_n r^{n-1}} \int_{\partial B_r(x_0)} [(\Delta)^{p-2}g] (y)dS \right),
\end{aligned}
\end{equation}
where $ \frac{\ptl}{\ptl r}$ is the radial derivative along the sphere.

Multiplying \eqref{c11} by $ \frac{r}{n}$ and integrating on $ (0,r)$, we obtain
\begin{equation}\label{c12}
\frac{r^2}{2n} [(\Delta)^{p-1}g](x_0) + [(\Delta)^{p-2}g](x_0) = \frac{1}{n\omega_n r^{n-1}} \int_{\partial B_r(x_0)} [(\Delta)^{p-2}g] (y)dS.
\end{equation}

Then multiplying \eqref{c12} by $ r^{n-1}n$ and integrating on $ (0,r)$ and dividing the whole resulting equation by $ r^n $ to get
\[
\frac{r^2}{2(n+2)} [(\Delta)^{p-1}g](x_0) + [(\Delta)^{p-2}g](x_0) = \frac{1}{\omega_n r^{n}} \int_{B_r(x_0)} [(\Delta)^{p-2}g] (y)dy.
\]
Repeating the above argument $ p-1$ times to get
\begin{equation}\label{c14}
\begin{aligned}
P(r):= & C_1(n,p)r^{2(p-1)} [(\Delta)^{p-1}g](x_0) + C_2(n,p)r^{2(p-2)} [(\Delta)^{p-2}g](x_0) \\
& + \cdots + C_{p-1}(n,p)r^{2} (\Delta g)(x_0) \\
 = &  \frac{1}{n\omega_n r^{n-1}} \int_{\partial B_r(x_0)} g(y)dS - g(x_0),
\end{aligned}
\end{equation}
where $ C_i(n,p) > 0$, $ i = 1, \cdots, p-1$.

By Jensen's inequality, one gets
\begin{equation}\label{c15}
\begin{aligned}
\text{exp}(\theta P(r)) & = e^{ -\theta g(x_0)} \text{exp}\left[\frac{1}{n\omega_n r^{n-1}} \int_{\partial B_r(x_0)} \theta g(y)dS \right] \\
& \leq e^{ - \theta g(x_0)} \frac{1}{n\omega_n r^{n-1}} \int_{\partial B_r(x_0)} e^{\theta g(y)}dS,
\end{aligned}
\end{equation}
where $ \theta$ is a constant.

Using \eqref{c13}, we see that
\[
g(x) = u(x) -J(x) \leq u(x) + \alpha \ln|x|.
\]
Let $ \theta = \frac{2n-\mu}{2}$ in \eqref{c15}. Then one gets
\begin{align*}
r^{-\frac{2n- \mu}{2}\alpha} e^{\frac{2n- \mu}{2} P(r)} & \leq r^{-\frac{2n- \mu}{2}\alpha} e^{ - \frac{2n- \mu}{2} g(x_0)} \frac{1}{n\omega_n r^{n-1}} \int_{\partial B_r(x_0)} e^{\frac{2n- \mu}{2} (u(y) + \alpha \ln|y|)}dS \\
& \leq C \frac{1}{r^{n-1}} \int_{\partial B_r(x_0)}  e^{\frac{2n- \mu}{2}u(y)} \left(\frac{|y|}{r} \right)^{\frac{2n- \mu}{2}\alpha} dS \\
& \leq C \frac{1}{r^{n-1}} \int_{\partial B_r(x_0)}  e^{\frac{2n- \mu}{2}u(y)}dS
\end{align*}
for $ r \geq 1$. Thus we have $ r^{-\frac{2n- \mu}{2}\alpha} e^{\frac{2n- \mu}{2} P(r)} \in L^1(1,\infty)$. Hence the leading coefficient in $ P(r)$ must be non-positive. That is,
\[
[(\Delta)^{p-1}g](x_0) = -C_0 \leq 0.
\]

Note that by the definition of $ J(x)$, we know that $ \Delta J(x) \leq 0$ in $ \mathbb{R}^n$. By mean value property for super-harmonic function, we have for any $ x_0\in \mathbb{R}^n$ and $ r> 0$
\begin{equation}\label{c17}
J(x_0)\geq \frac{1}{n\omega_n r^{n-1}} \int_{\partial B_r(x_0)} J(y)dS.
\end{equation}
Then it follows from \eqref{c14} and \eqref{c17} that
\[
P(r) \geq  \frac{1}{n\omega_n r^{n-1}} \int_{\partial B_r(x_0)} u(y)dS - u(x_0).
\]
Since $ u=o(|x|^2)$ at $ \infty$, by dividing both sides by $ r^2$ and letting $ r \rightarrow \infty$, we observe that
\[
\varliminf_{r \rightarrow \infty} \frac{P(r)}{r^2} \geq 0.
\]
If $ C_0 < 0$, then we would have
\[
\varlimsup_{r \rightarrow \infty} \frac{P(r)}{r^2} <0
\]
or $ \frac{P(r)}{r^2}$ tends to $ -\infty$ as $ r\rightarrow \infty$. This is impossible. We deduce that
\[
[(\Delta)^{p-1}g](x) = 0, \quad \forall x \in \mathbb{R}^n.
\]

Similarly, $ (\Delta)^{p-2}g = \cdots = \Delta g =0$, i.e.,
\begin{equation}\label{c67}
\Delta u =\Delta J \leq 0.
\end{equation}
If $ p=1$ and $ n=2$, we get directly \eqref{c67}.
According to \eqref{c3}, \eqref{c13} and Lemma \ref{c25}, we finally conclude that $ g=u-J$ is an entire harmonic function with
\[
\liminf_{|x| \rightarrow \infty} \frac{J(x) - u(x)}{|x|^2} \ge 0,
\]
for $ n\geq 2$. Then the classical Liouville theorem implies that
\[
u=J(x)+ \sum_{i=1}^n c_i x_i + d
\]
for some constants $ c_i$ and $ d$. Now the integrability condition \eqref{c10} forces $ c_i = 0$ for $ i=1,2,\cdots,n.$ Hence, we conclude that $ u$ satisfies the integral equation \eqref{c29}.
\end{proof}

\begin{proof}[Proof of Theorem \ref{c61}] By Lemma \ref{c68} and Theorem \ref{c28}, we deduce that
\begin{equation}\label{c69}
\lim_{|x| \rightarrow \infty} \frac{u(x)}{ \ln |x|} = -\alpha.
\end{equation}

Since
\[
\int_{\mathbb{R}^n}e^{\frac{2n- \mu}{2}u(y)} dy < \infty,
\]
from \eqref{c13} and \eqref{c29}, we conclude that
\begin{equation}\label{c55}
\alpha > \frac{2n}{2n- \mu}.
\end{equation}

Then we claim that
\[
|x|^{\mu} v(x) - \beta \rightarrow 0 \ \ \text{as} \ \ |x| \rightarrow \infty.
\]
We know from
\[
|x|^{\mu} v(x) - \beta = \int_{\mathbb{R}^n} \frac{|x|^{\mu} - |x-y|^\mu}{|x-y|^\mu} e^{\frac{2n- \mu}{2}u(y)} dy.
\]

We take $ A_i$, $ i = 1,2,3$ as in Lemma \ref{c68}. If $ y\in A_1$, one has
\[
\frac{|x|^{\mu} - |x-y|^\mu}{|x-y|^\mu} \rightarrow 0
\]
as $ |x| \rightarrow \infty$. Hence,
\[
\int_{A_1} \frac{|x|^{\mu} - |x-y|^\mu}{|x-y|^\mu} e^{\frac{2n- \mu}{2}u(y)} dy \rightarrow 0 \ \ \text{as} \ \ |x| \rightarrow \infty.
\]

Next, if $ y\in A_2$, we have $ \frac{|x|}{2} \le |y| \le \frac{3}{2}|x|$. According to \eqref{c69} and \eqref{c55}, we deduce that
\begin{equation}
\begin{aligned}
& \left| \int_{A_2} \frac{|x|^{\mu} - |x-y|^\mu}{|x-y|^\mu} e^{\frac{2n- \mu}{2}u(y)} dy \right| \\
\le & \int_{A_2} e^{\frac{2n- \mu}{2}u(y)} dy + |x|^{\mu} \int_{A_2} \frac{e^{\frac{2n- \mu}{2}u(y)}}{|x-y|^\mu} dy \\
\le & o(1) + |x|^{\mu - \frac{2n- \mu}{2}(\alpha - o(1))} \int_{A_2} \frac{1}{|x-y|^\mu}dy \\
\le & o(1) + C|x|^{n - \frac{2n- \mu}{2}\alpha + o(1)} \rightarrow 0
\nonumber
\end{aligned}
\end{equation}
as $ |x| \rightarrow \infty$.

Finally, if $ y \in A_3$, we have
\[
\frac{|x|^{\mu}}{|x-y|^\mu} \le 2^{\mu},
\]
which implies
\[
\left| \int_{A_3} \frac{|x|^{\mu} - |x-y|^\mu}{|x-y|^\mu} e^{\frac{2n- \mu}{2}u(y)} dy \right| \le C \int_{A_3}e^{\frac{2n- \mu}{2}u(y)} dy \le C\eps.
\]
In fact, we are going to give a more precise decay of $ u(x)$ as following
\begin{equation}\label{c19}
u(x) = -\alpha \ln|x| + O(1),
\end{equation}
for $ |x|$ large enough. We find
\[
  u(x) + \alpha \ln|x| = \frac{1}{\beta_n}  \int_{\mathbb{R}^n}\left[ \ln \left(\frac{|x|(|y|+1)}{|x-y|}\right)\right] v(y) e^{\frac{2n-\mu}{2}u(y)} dy + C.
\]
A direct calculation yields that
\begin{equation}
  \begin{split}
& \int_{A_2}\left[ \ln \left(\frac{|x|(|y|+1)}{|x-y|}\right)\right] v(y) e^{\frac{2n-\mu}{2}u(y)} dy \\
\leq & C\int_{A_2} (\ln (2|x|^2)) v(y) e^{\frac{2n-\mu}{2}u(y)} dy + C|x|^{o(1)-\mu - \alpha\frac{2n-\mu}{2} }\int_{A_2} \ln \frac{1}{|x-y|} dy \\
\leq & C|x|^{n+o(1)-\mu - \alpha\frac{2n-\mu}{2}}(\ln (\sqrt{2}|x|) + \ln|x|)\\
\leq & C, \nonumber
    \end{split}
\end{equation}
for $ |x|$ large enough. For $ x\in A_1 \cup A_3$, it is clear that
\begin{equation}
  \begin{aligned}
& \int_{A_1 \cup A_3}\left[ \ln \left(\frac{|x|(|y|+1)}{|x-y|}\right)\right] v(y) e^{\frac{2n-\mu}{2}u(y)} dy \\
\leq & \int_{A_1 \cup A_3} (\ln 2(|y|+1) )v(y) e^{\frac{2n-\mu}{2}u(y)} dy\\
\leq & C\int_{A_1} v(y) e^{\frac{2n-\mu}{2}u(y)} dy + C\int_{A_3} \frac{ \ln (2(|y|+1))}{|y|^\mu} e^{\frac{2n-\mu}{2}u(y)} dy \\
\leq & C.
  \nonumber
    \end{aligned}
\end{equation}
We conclude that
\[
u(x) + \alpha \ln|x| \leq C.
\]
By \eqref{c13}, we obtain \eqref{c19}. This completes the proof.
\end{proof}

\begin{lemma}\label{c58}
Suppose $ u(x)\in  C^n$ is a solution of \eqref{c29} with $ n\geq 2$. And if one sets
\be\label{c31}
\alpha:= \frac{1}{\beta_n}  \int_{\mathbb{R}^n}v(y) e^{\frac{2n-\mu}{2}u(y)} dy  < \infty, \ \int_{\mathbb{R}^n}e^{\frac{2n- \mu}{2}u(y)} dy < \infty,
\ee
then the following identity holds:
\begin{equation}\label{c27}
\alpha \left(\alpha - \frac{4n}{2n-\mu}\right) = \frac{4}{(2n-\mu)\beta_n} \int_{\mathbb{R}^n} \langle x,\nabla v(x) \rangle  e^{\frac{2n-\mu}{2}u(x)} dx.
\end{equation}
\end{lemma}
\begin{proof} Since $ u(x)\in C^n$, then $ \nabla u(x)$ is continuous. Differentiating equation \eqref{c29} yields
\begin{equation}\label{c30}
\langle x,\nabla u(x) \rangle = - \frac{1}{\beta_n} \int_{\mathbb{R}^n} \frac{\langle x,x-y \rangle}{|x-y|^2} v(y) e^{\frac{2n-\mu}{2}u(y)} dy.
\end{equation}

Multiplying both sides of \eqref{c30} by $ v(x) e^{\frac{2n-\mu}{2}u(x)}$ and integrating the resulting equation both sides over $ B_R$ for any $ R>0$, one gets
\begin{equation}\label{c32}
\begin{aligned}
& - \frac{1}{\beta_n} \int_{B_R} v(x) e^{\frac{2n-\mu}{2}u(x)}  \int_{\mathbb{R}^n} \frac{\langle x,x-y \rangle}{|x-y|^2} v(y) e^{\frac{2n-\mu}{2}u(y)} dydx \\
& = \int_{B_R} v(x) e^{\frac{2n-\mu}{2}u(x)} \langle x,\nabla u(x) \rangle dx.
\end{aligned}
\end{equation}
By a simple calculation, we obtain
\begin{align*}
& \int_{B_R} v(x) e^{\frac{2n-\mu}{2}u(x)} \langle x,\nabla u(x) \rangle dx \\
& = \frac{2}{2n-\mu} \int_{B_R} v(x) \langle x,\nabla e^{\frac{2n-\mu}{2}u(x)} \rangle dx \\
& = - \frac{2}{2n-\mu} \int_{B_R} e^{\frac{2n-\mu}{2}u(x)} (\langle x,\nabla v(x) \rangle + nv(x)) dx +\frac{2R}{2n-\mu} \int_{\ptl B_R} v(x) e^{\frac{2n-\mu}{2}u(x)} dS.
\end{align*}
By the asymptotic behavior of $ u$, $ v$ at $ \infty$, we have
\[
\lim_{R \rightarrow \infty} R \int_{\ptl B_{R}} v(x) e^{\frac{2n-\mu}{2}u(x)} dS = 0.
\]
We claim that
\[
\int_{\mathbb{R}^n} \left| v(x) e^{\frac{2n-\mu}{2}u(x)} \langle x,\nabla u(x) \rangle \right| dx < \infty \quad \text{and} \quad \int_{\mathbb{R}^n} \left| e^{\frac{2n-\mu}{2}u(x)} \langle x,\nabla v(x) \rangle \right| dx < \infty.
\]
Then, letting $ R \rightarrow \infty$, one has
\begin{equation}\label{c35}
\begin{aligned}
RHS & = - \frac{2}{2n-\mu} \int_{\mathbb{R}^n} e^{\frac{2n-\mu}{2}u(x)} (\langle x,\nabla v(x) \rangle + nv(x)) dx \\
& = - \frac{2}{2n-\mu} \int_{\mathbb{R}^n} e^{\frac{2n-\mu}{2}u(x)} \langle x,\nabla v(x) \rangle dx - \frac{2n}{2n-\mu} \beta_n \alpha.
\end{aligned}
\end{equation}
In order to prove this claim it is sufficient to prove
\begin{equation}\label{c57}
|x| |\nabla u(x)| \le C \quad \text{and} \quad |x| |\nabla v(x)| \le C .
\end{equation}
By \eqref{c30}, we know that
\[
|x| |\nabla u(x)| \le \frac{1}{\beta_n} \int_{\mathbb{R}^n} \frac{|x|}{|x-y|} v(y) e^{\frac{2n-\mu}{2}u(y)} dy.
\]
For $ |x|$ large, we divide $ \mathbb{R}^n$ into two parts:
\[
A_1 = \{y| |x-y| \le  \frac{|x|}{2}\}, \ A_2= \{y| |x-y| >  \frac{|x|}{2}\}.
\]
If $ y \in A_1$, by Theorem \ref{c61}, we deduce that
\begin{equation}
\begin{aligned}
& \int_{A_1} \frac{|x|}{|x-y|} v(y) e^{\frac{2n-\mu}{2}u(y)} dy \\
\le & C|x|^{1-\mu - \alpha \frac{2n-\mu}{2}}\int_{A_1} \frac{1}{|x-y|} dy \\
\le & C|x|^{n-\mu - \alpha \frac{2n-\mu}{2}} \rightarrow 0 \nonumber
\end{aligned}
\end{equation}
as $ |x| \rightarrow \infty$.

If $ y\in A_2$, it is clear that
\[
\int_{A_2} \frac{|x|}{|x-y|} v(y) e^{\frac{2n-\mu}{2}u(y)} dy \le 2 \int_{A_2} v(y) e^{\frac{2n-\mu}{2}u(y)} dy \le C.
\]
Thus, one gets $ |x| |\nabla u(x)| \le C$.

For $ |\nabla v(x)|$, we consider the case $ 0< \mu < n-1$. For $ |x|$ large enough, one has
\begin{equation}
  \begin{split}
  |x| |\nabla v(x)| & \leq \mu \int_{\mathbb{R}^n} \frac{|x| e^{\frac{2n-\mu}{2}u(y)}}{|x-y|^{\mu+1}} dy \\
  & = \mu \int_{A_1 } \frac{|x| e^{\frac{2n-\mu}{2}u(y)}}{|x-y|^{\mu+1}} dy + \mu \int_{A_2 } \frac{|x| e^{\frac{2n-\mu}{2}u(y)}}{|x-y|^{\mu+1}} dy \\
  & \leq C |x|^{1-\frac{2n-\mu}{2}\alpha } \int_{A_1 } \frac{1} {|x-y|^{\mu+1}} dy + C |x|^{-\mu} \int_{A_2 } e^{\frac{2n-\mu}{2}u(y)}dy \\
  & \leq C |x|^{n-\mu-\frac{2n-\mu}{2}\alpha }  + C |x|^{-\mu} \\
  & \leq C. \nonumber
  \end{split}
\end{equation}
For the case $ n-1 \leq \mu < n$, we can write
\[
\nabla v(x)=\frac{2n-\mu}{2} \int_{\mathbb{R}^n}\frac{e^{\frac{2n-\mu}{2} u(x-y)} \nabla_x u(x-y)}{|y|^\mu} dy.
\]
Then, for $ |x|$ large enough, we divide $ \mathbb{R}^n$ into four parts:
\[
D_1=\{y| |y| \geq 2|x|\}, \quad D_2 = \left\{y||y-x|>\frac{|x|}{2}, |y| \leq 2|x| \right\}
\]
\[
D_3 = \left\{y|K<|y-x| \leq \frac{|x|}{2}\right\}, \quad D_4 = \{y| |y-x| \leq K\},
\]
where $ K > 0$ is large. Theorem \ref{c61} implies that
\begin{equation}
  \begin{split}
    |x||\nabla v(x)|
& \leq \frac{2n-\mu}{2} \int_{\mathbb{R}^n}\frac{e^{\frac{2n-\mu}{2} u(x-y)} |\nabla_x u(x-y)| |x|}{|y|^\mu} dy \\
& \leq \frac{C}{|x|^{\mu-1}} \int_{D_1}e^{\frac{2n-\mu}{2} u(x-y)} |\nabla_x u(x-y)|  dy
+ \frac{C}{|x|^{\frac{2n-\mu}{2}\alpha }}\int_{D_2}\frac{1}{|y|^\mu} dy  \\
& \quad +  \frac{C}{|x|^{\mu-1}}  \int_{D_3 \cup D_4}e^{\frac{2n-\mu}{2} u(x-y)} |\nabla_x u(x-y)|  dy \\
& \leq  \frac{C}{|x|^{\mu-1}} \int_{|y-x| \geq |x|}e^{\frac{2n-\mu}{2} u(x-y)} |\nabla_x u(x-y)|  dy + \frac{C|x|^{n-\mu}}{|x|^{\frac{2n-\mu}{2}\alpha }} +\frac{C}{|x|^{\mu-1}} \\
& \leq \frac{C|x|^{n-\mu}}{|x|^{\frac{2n-\mu}{2}\alpha }} +\frac{C}{|x|^{\mu-1}}.
\nonumber
    \end{split}
\end{equation}
By \eqref{c55}, since $ n\geq 2$ and $ \mu
\geq 1$, we obtain
\[
|x||\nabla v(x)| \leq C.
\]
Therefore, we get \eqref{c57}.

With $ x = \frac{1}{2}((x-y)+(x+y))$, for the left hand side of \eqref{c32}, one has the following identity
\begin{equation}\label{c33}
\begin{aligned}
& - \frac{1}{\beta_n} \int_{B_R} v(x) e^{\frac{2n-\mu}{2}u(x)}  \int_{\mathbb{R}^n} \frac{\langle x,x-y \rangle}{|x-y|^2} v(y) e^{\frac{2n-\mu}{2}u(y)} dydx \\
& = - \frac{1}{2\beta_n} \int_{B_R} v(x) e^{\frac{2n-\mu}{2}u(x)} \int_{\mathbb{R}^n} v(y) e^{\frac{2n-\mu}{2}u(y)} dydx \\
& \quad - \frac{1}{2\beta_n} \int_{B_R} v(x) e^{\frac{2n-\mu}{2}u(x)} \int_{\mathbb{R}^n} \frac{\langle x+y,x-y \rangle}{|x-y|^2} v(y) e^{\frac{2n-\mu}{2}u(y)} dydx.
\end{aligned}
\end{equation}
We claim that
\[ \int_{\mathbb{R}^n}\int_{\mathbb{R}^n} \left| \frac{\langle x,x-y \rangle}{|x-y|^2} v(x) e^{\frac{2n-\mu}{2}u(x)} v(y) e^{\frac{2n-\mu}{2}u(y)} \right| dydx < \infty.
\]
Then, it is easy to see that the last term in \eqref{c33} will vanish when one takes the limit $ R \rightarrow \infty$ simply by changing variables $ x$ and $ y$. Thus the left hand side gives
\begin{equation}\label{c34}
LHS=-\frac{1}{2} \beta_n \alpha^2.
\end{equation}
A simple calculation yields
\begin{eqnarray*}
  &&\int_{\mathbb{R}^n}\int_{\mathbb{R}^n} \left| \frac{\langle x,x-y \rangle}{|x-y|^2} v(x) e^{\frac{2n-\mu}{2}u(x)} v(y) e^{\frac{2n-\mu}{2}u(y)} \right| dydx \\
  & \leq &\int_{\mathbb{R}^n} \int_{|y-x|\geq \frac{|x|}{2}} \frac{|x|}{|x-y|} v(x) e^{\frac{2n-\mu}{2}u(x)} v(y) e^{\frac{2n-\mu}{2}u(y)}  dydx + \int_{B_{R_0}} \int_{|y-x|< \frac{|x|}{2}} \frac{v(y) e^{\frac{2n-\mu}{2}u(y)} }{|x-y|} dydx \\
  && + \int_{B_{R_0}^c} \int_{|y-x|< \frac{|x|}{2}} \frac{|x|}{|x-y|} v(x) e^{\frac{2n-\mu}{2}u(x)} v(y) e^{\frac{2n-\mu}{2}u(y)}  dydx \\
  &\leq & C+ C \int_{B_{R_0}} \int_{|y-x|< \frac{|x|}{2}} \frac{1}{|x-y|} dydx + C\int_{B_{R_0}^c} \int_{|y-x|< \frac{|x|}{2}} \frac{v(x) e^{\frac{2n-\mu}{2}u(x)}}{|x-y|} |x|^{1-\mu -\alpha \frac{2n-\mu}{2}}dydx \\
  &\leq & C +  C\int_{B_{R_0}^c}  v(x) e^{\frac{2n-\mu}{2}u(x)} |x|^{n-\mu -\alpha \frac{2n-\mu}{2}}dx \\
 & \leq & C,
\end{eqnarray*}
for some large constant $ R_0 > 0$. Hence, \eqref{c27} follows from \eqref{c35} and \eqref{c34}.
\end{proof}

Using the asymptotic behavior of $ u$, $ v$ at $ \infty$, we work out the exact value of $ \alpha$.
\begin{lemma}\label{c65}
Suppose $ u(x)\in  C^n$ is a solution of \eqref{c29} with $ n \geq 2$. And if one sets
\[
\alpha:= \frac{1}{\beta_n}  \int_{\mathbb{R}^n}v(y) e^{\frac{2n-\mu}{2}u(y)} dy  < \infty, \ \int_{\mathbb{R}^n}e^{\frac{2n- \mu}{2}u(y)} dy < \infty,
\]
then
\[
\alpha =2.
\]
\end{lemma}
\begin{proof}
Let
\[
v_{\eps}(x) := \int_{\mathbb{R}^n\setminus B_{\eps}(x)} \frac{e^{\frac{2n-\mu}{2}u(y)}}{|x-y|^{\mu}} dy =  \int_{\mathbb{R}^n\setminus B_{\eps}(0)} \frac{e^{\frac{2n-\mu}{2}u(x-y)}}{|y|^{\mu}} dy.
\]
We claim that
\[
\lim_{\eps \rightarrow 0^+}  \int_{\mathbb{R}^n} \langle x,\nabla ( v(x)-v_{\eps}(x)) \rangle  e^{\frac{2n-\mu}{2}u(x)} dx =0.
\]
By using Lemma \ref{c25}, for any $ \varphi(x) \in C_c^\infty(\mathbb{R}^n)$, we conclude that
\begin{equation}\label{c54}
\begin{aligned}
&\lim_{\eps \rightarrow 0^+}\int_{\mathbb{R}^n} (v(x)-v_{\eps}(x)) |\nabla \varphi(x)| dx \\
= & \lim_{\eps \rightarrow 0^+}\int_{\mathbb{R}^n} \int_{B_{\eps}(x)} \frac{e^{\frac{2n-\mu}{2}u(y)}}{|x-y|^{\mu}} dy |\nabla \varphi(x)| dx \\
\le& C\lim_{\eps \rightarrow 0^+}\int_{supp \varphi} \int_{B_{\eps}(x)} \frac{1}{|x-y|^{\mu}} dy dx \\
= & 0.
\end{aligned}
\end{equation}

Take a cutoff function $ \eta_R(x) \in C_c^\infty(B_{2R})$ and $ 0 \le \eta_R(x) \le 1$ such that
\begin{equation}
\eta_R(x) =
\left\{
\begin{aligned}
& 1, \quad x\in B_R \\
& 0, \quad x \in B_{2R}^c,
\end{aligned}
\right. \nonumber
\end{equation}
where $ R$ is large enough. By \eqref{c54}, we get
\[
\lim_{\eps \rightarrow 0^+} \int_{\mathbb{R}^n} \langle x,\nabla ( v(x)-v_{\eps}(x)) \rangle  e^{\frac{2n-\mu}{2}u(x)} \eta_R(x) dx =0.
\]
For $ |x| \ge R$, by \eqref{c57}, a simple calculation yields
\begin{equation}
\begin{aligned}
| \langle x, \nabla (v(x)-v_{\eps}(x)) \rangle | = & \left| \frac{2n-\mu}{2} \int_{B_\eps} \frac{ \langle x, \nabla_x u(x-y) \rangle e^{\frac{2n-\mu}{2}u(x-y)} }{|y|^\mu} dy  \right| \\
\le & C |x|^{-\alpha \frac{2n-\mu}{2}} \int_{B_\eps} \frac{1}{|y|^{\mu}} dy \\
\le & C\eps^{n-\mu} |x|^{-\alpha \frac{2n-\mu}{2}}.
\nonumber
\end{aligned}
\end{equation}
Then, we deduce that
\begin{equation}
\begin{aligned}
& \int_{\mathbb{R}^n \setminus B_R}\left|  \langle x,\nabla ( v(x)-v_{\eps}(x)) \rangle  e^{\frac{2n-\mu}{2}u(x)} (1-\eta_R(x)) \right| dx \\
\le & C\eps^{n-\mu} \int_{\mathbb{R}^n \setminus B_R}  |x|^{-\alpha(2n-\mu)} dx \\
\le & C \eps^{n-\mu} R^{n-\alpha(2n-\mu)}.
\nonumber
\end{aligned}
\end{equation}
Thus, one obtains
\begin{equation}
\begin{aligned}
& \lim_{\eps \rightarrow 0^+} \int_{\mathbb{R}^n} \langle x,\nabla ( v(x)-v_{\eps}(x)) \rangle  e^{\frac{2n-\mu}{2}u(x)} dx \\
=&  \lim_{\eps \rightarrow 0^+}  \int_{\mathbb{R}^n} \langle x,\nabla ( v(x)-v_{\eps}(x)) \rangle  e^{\frac{2n-\mu}{2}u(x)}( \eta_R(x) +(1-\eta_R(x))) dx  \\
=& 0.
\nonumber
\end{aligned}
\end{equation}
That is
\begin{equation}\label{c64}
\int_{\mathbb{R}^n} \langle x,\nabla v(x) \rangle  e^{\frac{2n-\mu}{2}u(x)} dx =  \lim_{\eps \rightarrow 0^+} \int_{\mathbb{R}^n} \langle x,\nabla v_{\eps}(x) \rangle  e^{\frac{2n-\mu}{2}u(x)} dx.
\end{equation}
Now, we give the representation formula for $ \nabla v_{\eps}(x)$. We have
\begin{equation}
\begin{aligned}
\nabla v_{\eps}(x) = & \nabla_x \left( \int_{\mathbb{R}^n\setminus B_{\eps}} \frac{e^{\frac{2n-\mu}{2}u(x-y)}}{|y|^{\mu}} dy \right) \\
= & \int_{\mathbb{R}^n\setminus B_{\eps}} \frac{\nabla_x \left( e^{\frac{2n-\mu}{2}u(x-y)} \right) }{|y|^{\mu}} dy \\
= &\int_{\mathbb{R}^n\setminus B_{\eps}(x)} \frac{\nabla \left( e^{\frac{2n-\mu}{2}u(y)}\right)}{|x-y|^{\mu}} dy \\
= & \lim_{R \rightarrow \infty} \int_{B_R(x) \setminus B_{\eps}(x)} \frac{\nabla \left( e^{\frac{2n-\mu}{2}u(y)}\right)}{|x-y|^{\mu}} dy \\
= & \lim_{R \rightarrow \infty} \int_{\ptl B_R(x) } \frac{ e^{\frac{2n-\mu}{2}u(y)} \nu}{|x-y|^{\mu}} dS - \int_{\ptl B_\eps(x) } \frac{ e^{\frac{2n-\mu}{2}u(y)} \nu}{|x-y|^{\mu}} dS \\
& -\mu \lim_{R \rightarrow \infty} \int_{B_R(x) \setminus B_{\eps}(x)} \frac{ e^{\frac{2n-\mu}{2}u(y)} (x-y)}{|x-y|^{\mu+2}} dy \\
:= & V^1 - V_\eps^2 - V_\eps^3,
\nonumber
\end{aligned}
\end{equation}
where $ \nu = \frac{y-x}{|y-x|}$ is the unit outer normal direction. For the integral $ V^1$, there holds
\[
\lim_{R \rightarrow \infty}  \int_{\ptl B_R(x) } \left| \frac{ e^{\frac{2n-\mu}{2}u(y)} \nu}{|x-y|^{\mu}} \right|dS \le C \lim_{R \rightarrow \infty} \frac{1}{R^{\mu}} \int_{\ptl B_R(x) } e^{\frac{2n-\mu}{2}u(y)} dy  =0.
\]
Thus we have $ \nabla v_{\eps}(x) = -V^2_{\eps}-V^3_{\eps} $. For the integral $ V_\eps^2$, there exists a $ \xi(y) \in B_{\eps}(x)$ such that
\begin{equation}
\begin{aligned}
\int_{\ptl B_\eps(x) } \frac{ e^{\frac{2n-\mu}{2}u(y)} \nu}{|x-y|^{\mu}} dS = & \int_{\ptl B_\eps(x) } \frac{ e^{\frac{2n-\mu}{2}u(y)} - e^{\frac{2n-\mu}{2}u(x)} }{|x-y|^{\mu}} \cdot\frac{y-x}{|y-x|} dS \\
= & \frac{2n-\mu}{2} \int_{\ptl B_\eps(x) } \frac{ e^{\frac{2n-\mu}{2}u(\xi(y))}   \langle \nabla u(\xi(y)), y-x  \rangle }{|x-y|^{\mu}} \cdot\frac{y-x}{|y-x|} dS. \\
\nonumber
\end{aligned}
\end{equation}
If $ |x| \le R_0$ and $ R_0$ is large, there exists a constant $ C>0 $ such that
\begin{equation}
\begin{aligned}
& \int_{\ptl B_\eps(x) } \left| \frac{ e^{\frac{2n-\mu}{2}u(\xi(y))}   \langle \nabla u(\xi(y)), y-x  \rangle }{|x-y|^{\mu}} \cdot\frac{y-x}{|y-x|} \right| dS \\
\le & C \int_{\ptl B_\eps(x) } \frac{1}{|x-y|^{\mu-1}} dS \\
\le & C \eps^{n-\mu}.
\nonumber
\end{aligned}
\end{equation}
If $ |x| \ge R_0$, we get
\begin{equation}
\begin{aligned}
& \int_{\ptl B_\eps(x) } \left| \frac{ e^{\frac{2n-\mu}{2}u(\xi(y))}   \langle \nabla u(\xi(y)), y-x  \rangle }{|x-y|^{\mu}} \cdot\frac{y-x}{|y-x|} \right| dS \\
\le & C |x|^{-1-\alpha \frac{2n-\mu}{2} } \int_{\ptl B_\eps(x) } \frac{1}{|x-y|^{\mu-1}} dS \\
\le & C \eps^{n-\mu}|x|^{-1-\alpha \frac{2n-\mu}{2} }.
\nonumber
\end{aligned}
\end{equation}
By a simple calculation, we find
\begin{equation}\label{c66}
\begin{aligned}
& \lim_{\eps \rightarrow 0^+} \left| \int_{\mathbb{R}^n} \langle x, V^2_\eps \rangle  e^{\frac{2n-\mu}{2}u(x)} dx \right| \\
=&\lim_{\eps \rightarrow 0^+}\left| \int_{\mathbb{R}^n} \frac{2n-\mu}{2} \int_{\ptl B_\eps(x) } \frac{ e^{\frac{2n-\mu}{2}u(\xi(y))}   \langle \nabla u(\xi(y)), y-x  \rangle }{|x-y|^{\mu}} \cdot\frac{ \langle x,y-x\rangle}{|y-x|} dS e^{\frac{2n-\mu}{2}u(x)} dx \right|\\
\le & C \lim_{\eps \rightarrow 0^+} \int_{B_{R_0}}  \eps^{n-\mu}|x| e^{\frac{2n-\mu}{2}u(x)}  dx
+ C\lim_{\eps \rightarrow 0^+} \int_{B^c_{R_0}} \eps^{n-\mu}|x|^{-\alpha \frac{2n-\mu}{2} }   e^{\frac{2n-\mu}{2}u(x)} dx \\
=& 0.
\end{aligned}
\end{equation}
Therefore, with \eqref{c64} and \eqref{c66}, one has
\begin{equation}
\begin{aligned}
& \int_{\mathbb{R}^n} \langle x,\nabla v(x) \rangle  e^{\frac{2n-\mu}{2}u(x)} dx \\
= & -\mu  \lim_{\eps \rightarrow 0^+} \int_{\mathbb{R}^n} \int_{\mathbb{R}^n\setminus B_{\eps}(x)} \frac{\langle x,x-y \rangle e^{\frac{2n-\mu}{2}u(y)} }{|x-y|^{\mu+2}} dy e^{\frac{2n-\mu}{2}u(x)} dx \\
= & -\mu \lim_{\eps \rightarrow 0^+} \int_{\mathbb{R}^n}  \int_{\mathbb{R}^n\setminus B_{\eps}(x)} \frac{\langle x,x-y \rangle \left( e^{\frac{2n-\mu}{2}u(y)} - \chi_{B_{\delta}(x)}(y) e^{\frac{2n-\mu}{2}u(x)} \right)}{|x-y|^{\mu+2}} dy  e^{\frac{2n-\mu}{2}u(x)} dx
\nonumber
\end{aligned}
\end{equation}
for $ \delta > 0$ small enough. For any $ \Omega \subset \mathbb{R}^n$, we define
\begin{equation}
\chi_{\Omega}(x) =
\left\{
\begin{aligned}
& 1, \quad x\in \Omega \\
& 0, \quad x \in \Omega^c.
\end{aligned}
\right. \nonumber
\end{equation}

We first prove that the above integral is absolutely integrable in $ \mathbb{R}^n \times \mathbb{R}^n$. We consider that
\begin{equation}
\begin{aligned}
& \int_{\mathbb{R}^n}\int_{\mathbb{R}^n} \left| \frac{\langle x,x-y \rangle \left( e^{\frac{2n-\mu}{2}u(y)} - \chi_{B_{\delta}(x)}(y) e^{\frac{2n-\mu}{2}u(x)} \right)}{|x-y|^{\mu+2}} e^{\frac{2n-\mu}{2}u(x)}  \right| dy dx \\
\le & \int_{\mathbb{R}^n}\int_{B_{\delta}(x)} \frac{|x| \left| e^{\frac{2n-\mu}{2}u(y)} - e^{\frac{2n-\mu}{2}u(x)} \right|}{|x-y|^{\mu+1}} e^{\frac{2n-\mu}{2}u(x)} dydx + \int_{\mathbb{R}^n}\int_{ B^c_{\delta}(x)} \frac{|x| e^{\frac{2n-\mu}{2}u(y)} }{|x-y|^{\mu+1}} e^{\frac{2n-\mu}{2}u(x)} dydx \\
:= & I_1 + I_2
\nonumber
\end{aligned}
\end{equation}
By \eqref{c55}, \eqref{c19} and \eqref{c57}, we observe that
\begin{equation}
\begin{aligned}
I_1 \le & C\int_{\mathbb{R}^n}\int_{B_{\delta}(x)} \frac{|x||\nabla u(\xi(y))| e^{\frac{2n-\mu}{2}u(\xi(y))}}{|x-y|^{\mu}} e^{\frac{2n-\mu}{2}u(x)} dydx \\
\le & C\int_{B_{R_0}}\int_{B_{\delta}(x)} \frac{ e^{\frac{2n-\mu}{2}u(x)}}{|x-y|^{\mu}} dydx + C\int_{B^c_{R_0}}\int_{B_{\delta}(x)} \frac{|x|^{ - \frac{2n-\mu}{2}\alpha}}{|x-y|^{\mu}} e^{\frac{2n-\mu}{2}u(x)} dydx\\
\le & C\delta^{n-\mu} \int_{B_{R_0}}e^{\frac{2n-\mu}{2}u(x)}dx + C\delta^{n-\mu} \int_{B^c_{R_0}}|x|^{ - (2n-\mu)\alpha} dx \\
\le & C,
\nonumber
\end{aligned}
\end{equation}
where $ \xi(y) \in B_{\delta}(x)$. We divide $ \mathbb{R}^n \setminus B_{\delta}(x)$ into two parts:
\[
A_1 = \left\{y | \delta \le |x-y| \le  \frac{|x|}{2}\right\}, \ A_2= \left\{y| |x-y| >  \frac{|x|}{2}\right\}.
\]
For $ y\in A_1$, We deduce that
\begin{equation}
\begin{aligned}
& \int_{\mathbb{R}^n}\int_{A_1} \frac{|x| e^{\frac{2n-\mu}{2}u(y)} }{|x-y|^{\mu+1}} e^{\frac{2n-\mu}{2}u(x)} dydx \\
\le & \int_{B_{R_0}}\int_{A_1} \frac{C}{|x-y|^{\mu+1}}dy |x| e^{\frac{2n-\mu}{2}u(x)} dx  + \int_{B^c_{R_0}}\int_{A_1} \frac{C}{|x-y|^{\mu+1}}dy |x|^{1-\alpha (2n-\mu)} dx \\
\le & C \int_{B_{R_0}} |x|^{n-\mu} e^{\frac{2n-\mu}{2}u(x)} dx + C \delta^{n-\mu-1}\int_{B_{R_0}} |x| e^{\frac{2n-\mu}{2}u(x)} dx + \int_{B^c_{R_0}} |x|^{n-\mu -\alpha (2n-\mu)} dx \\
& + C \delta^{n-\mu-1}\int_{B^c_{R_0}} |x|^{1 -\alpha (2n-\mu)} dx \\
\le & C + C \delta^{n-\mu-1} \\
\le & C
\nonumber
\end{aligned}
\end{equation}
for $ R_0$ large enough. In order to obtain the above inequality, we used the asymptotic behavior of $ u$ at $ \infty$ and $ \alpha > \frac{2n}{2n- \mu}$. For $ y \in A_2$, one gets
\begin{equation}
\begin{aligned}
\int_{\mathbb{R}^n}\int_{A_2} \frac{|x| e^{\frac{2n-\mu}{2}u(y)} }{|x-y|^{\mu+1}} e^{\frac{2n-\mu}{2}u(x)} dydx
\le & 2 \int_{\mathbb{R}^n}\int_{A_2} \frac{ e^{\frac{2n-\mu}{2}u(y)} }{|x-y|^{\mu}} dy e^{\frac{2n-\mu}{2}u(x)} dx \\
\le & 2 \int_{\mathbb{R}^n}v(x) e^{\frac{2n-\mu}{2}u(x)} dx \\
\le & C.
\nonumber
\end{aligned}
\end{equation}
Hence,
\[
I_1+ I_2 \le C.
\]
We consider that
\begin{equation}\label{c60}
\begin{aligned}
& \int_{\mathbb{R}^n} \langle x,\nabla v(x) \rangle  e^{\frac{2n-\mu}{2}u(x)} dx \\
= & -\mu \int_{\mathbb{R}^n} \int_{\mathbb{R}^n} \frac{\langle x,x-y \rangle \left( e^{\frac{2n-\mu}{2}u(y)} - \chi_{B_{\delta}(x)}(y) e^{\frac{2n-\mu}{2}u(x)} \right)}{|x-y|^{\mu+2}} e^{\frac{2n-\mu}{2}u(x)} dy dx \\
= & -\mu \int_{\mathbb{R}^n} \int_{\mathbb{R}^n} \frac{ e^{\frac{2n-\mu}{2}u(y)} - \chi_{B_{\delta}(x)}(y) e^{\frac{2n-\mu}{2}u(x)} }{|x-y|^{\mu}} e^{\frac{2n-\mu}{2}u(x)} dy dx \\
& -\mu \int_{\mathbb{R}^n} \int_{\mathbb{R}^n} \frac{\langle y,x-y \rangle \left( e^{\frac{2n-\mu}{2}u(y)} - \chi_{B_{\delta}(x)}(y) e^{\frac{2n-\mu}{2}u(x)} \right)}{|x-y|^{\mu+2}} e^{\frac{2n-\mu}{2}u(x)} dy dx \\
:=& II_1 + II_2
\end{aligned}
\end{equation}

Using the Lebesgue's dominated convergence theorem for $ II_1$, one can easily deduce that
\begin{equation}\label{c62}
\begin{aligned}
\lim_{\delta \rightarrow 0^+}\int_{\mathbb{R}^n} \int_{\mathbb{R}^n} \frac{ e^{\frac{2n-\mu}{2}u(y)} - \chi_{B_{\delta}(x)}(y) e^{\frac{2n-\mu}{2}u(x)} }{|x-y|^{\mu}} e^{\frac{2n-\mu}{2}u(x)} dy dx = & \int_{\mathbb{R}^n} \int_{\mathbb{R}^n} \frac{ e^{\frac{2n-\mu}{2}u(y)}}{|x-y|^{\mu}} e^{\frac{2n-\mu}{2}u(x)} dy dx \\
=& \beta_n \alpha
\end{aligned}
\end{equation}
We prove that the dominant function of $ II_1$ is integrable. That is
\begin{equation}
\begin{aligned}
& \int_{\mathbb{R}^n} \int_{\mathbb{R}^n} \left( \frac{ e^{\frac{2n-\mu}{2}u(y)}e^{\frac{2n-\mu}{2}u(x)} }{|x-y|^{\mu}} + \frac{\chi_{B_{1}(x)}(y) e^{(2n-\mu)u(x)} }{|x-y|^{\mu}} \right) dy dx \\
=& \int_{\mathbb{R}^n} v(x) e^{\frac{2n-\mu}{2}u(x)}dx + \int_{\mathbb{R}^n} \int_{B_{1}(x)} \frac{ e^{(2n-\mu)u(x)} }{|x-y|^{\mu}} dy dx \\
\le & C  + C \int_{\mathbb{R}^n} e^{(2n-\mu)u(x)} dx \\
\le & C.
\nonumber
\end{aligned}
\end{equation}
By a simple calculation, we obtain
\begin{equation}
\begin{aligned}
II_2
= & \mu \int_{\mathbb{R}^n} \int_{\mathbb{R}^n} \frac{\langle y,y-x \rangle \left( e^{\frac{2n-\mu}{2}u(y)} - \chi_{B_{\delta}(x)}(y) e^{\frac{2n-\mu}{2}u(x)} \right)}{|x-y|^{\mu+2}} \left( e^{\frac{2n-\mu}{2}u(x)} - \chi_{B_{\delta}(y)}(x) e^{\frac{2n-\mu}{2}u(y)}  \right) dy dx \\
& -\mu \int_{\mathbb{R}^n} \int_{\mathbb{R}^n} \frac{\langle y,x-y \rangle \left( e^{\frac{2n-\mu}{2}u(y)} - \chi_{B_{\delta}(x)}(y) e^{\frac{2n-\mu}{2}u(x)} \right)}{|x-y|^{\mu+2}} \chi_{B_{\delta}(y)}(x) e^{\frac{2n-\mu}{2}u(y)}  dy dx \\
= & -\mu \int_{\mathbb{R}^n} \int_{\mathbb{R}^n} \frac{\langle y,x-y \rangle \left( e^{\frac{2n-\mu}{2}u(x)} - \chi_{B_{\delta}(y)}(x) e^{\frac{2n-\mu}{2}u(y)}  \right)}{|x-y|^{\mu+2}} e^{\frac{2n-\mu}{2}u(y)} dy dx \\
& +\mu \int_{\mathbb{R}^n} \int_{\mathbb{R}^n} \frac{\langle y,x-y \rangle \left( \chi_{B_{\delta}(x)}(y) e^{(2n-\mu)u(x)} - \chi_{B_{\delta}(y)}(x) e^{(2n-\mu)u(y)}  \right)}{|x-y|^{\mu+2}} dy dx. \\
:=&-III_1+III_2.
\nonumber
\end{aligned}
\end{equation}
According to \eqref{c60}, we know that
\[
III_1= II_1 + II_2.
\]

Similarly, we deduce that the integral $ III_2$ is absolutely integrable
\[
\mu \int_{\mathbb{R}^n} \int_{\mathbb{R}^n} \left| \frac{\langle y,x-y \rangle \chi_{B_{1}(x)}(y) \left(e^{(2n-\mu)u(x)} - e^{(2n-\mu)u(y)}  \right)}{|x-y|^{\mu+2}} \right| dy dx < \infty.
\]
Using the Lebesgue's dominated convergence theorem, we have
\begin{equation}\label{c63}
\lim_{\delta \rightarrow 0^+ }\int_{\mathbb{R}^n} \int_{\mathbb{R}^n} \frac{\langle y,x-y \rangle \chi_{B_{\delta}(x)}(y) \left(e^{(2n-\mu)u(x)} - e^{(2n-\mu)u(y)}  \right)}{|x-y|^{\mu+2}} dy dx = 0.
\end{equation}
Therefore, by \eqref{c60}-\eqref{c63}, we obtain
\[
\int_{\mathbb{R}^n} \langle x,\nabla v(x) \rangle  e^{\frac{2n-\mu}{2}u(x)} dx = -\frac{\mu \beta_n \alpha}{2}.
\]
By Lemma \ref{c58}, we deduce that
\[
\alpha \left(\alpha - \frac{4n}{2n-\mu}\right) = - \frac{4}{(2n-\mu)\beta_n} \times\frac{\mu \beta_n \alpha}{2}.
\]
Thus, one gets
\[
\alpha = 2.
\]
\end{proof}

\section{Moving sphere}
In this section, we prove our main results by moving spheres method.

It follows from Theorem \ref{c28} that
\[
u(x) = \frac{1}{\beta_n}  \int_{\mathbb{R}^n}\left[ \ln \left(\frac{|y|+1}{|x-y|}\right)\right] v(y) e^{\frac{2n-\mu}{2}u(y)} dy + C.
\]
Clearly $ u$ satisfies the assumptions $ (a)$ and $ (b)$ of Lemma \ref{c25}. Therefore,
\[
u(x) \le C, \quad x\in \mathbb{R}^n,
\]
for some constant $ C>0$. By Lemma \ref{c65}, we obtain the precise decay of $ u$
\begin{equation}\label{c70}
u(x) = -2\ln|x| + O(1),  \quad \text{for } \quad  |x|>>1.
\end{equation}

With these preparations, we are ready to run the method of moving spheres. We consider the Kelvin transformation of $ (u,v)$:
\begin{equation}
\left\{
\begin{aligned}
& p(x) = u\left(\frac{x}{|x|^2}\right) - 2\ln |x|, \\
& q(x) = \frac{1}{|x|^{\mu}}v\left( \frac{x}{|x|^2}\right).
\end{aligned}
\right.
\end{equation}
Then a direct calculation shows that they satisfy the following equation
\begin{equation}\label{c42}
\left\{
\begin{aligned}
& p(x) = \frac{1}{\beta_n}  \int_{\mathbb{R}^n}\left[ \ln \left(\frac{|y|+1}{|x-y|}\right)\right] q(y) e^{\frac{2n-\mu}{2}p(y)} dy + C \\
& q(x)=  \int_{\mathbb{R}^n} \frac{e^{\frac{2n-\mu}{2}p(y)}}{|x-y|^{\mu}} dy
\end{aligned}
\right. \text{for} \  x\neq 0.
\end{equation}
We use the moving spheres method on $ (p,q)$. We define
\begin{equation}
\left\{
\begin{aligned}
& p_{\lam}(x) = p\left(\frac{\lam^2 x}{|x|^2}\right) +2\ln \frac{\lam}{|x|} =u\left(\frac{x}{\lam^2}\right) - 2\ln \lam, \\
& q_{\lam}(x) = \frac{\lam^\mu}{|x|^{\mu}}q\left(\frac{\lam^2 x}{|x|^2}\right) = \frac{1}{\lam^{\mu}}v\left( \frac{x}{\lam^2}\right).
\end{aligned}
\right.
\end{equation}
Let
\[
w_{\lam}(x) = p(x) - p_{\lam}(x), \quad z_{\lam}(x) = q(x) - q_{\lam}(x).
\]
Then we have
\begin{equation}\label{c38}
z_{\lam}(x) = \int_{B_\lam} \left( \frac{1}{|x-y|^\mu} - \frac{1}{|x-\frac{\lam^2 y}{|y|^2}|^\mu}\frac{\lam^\mu}{|y|^\mu}\right) \left( e^{\frac{2n-\mu}{2}p(y)} - e^{\frac{2n-\mu}{2}p_\lam(y)} \right) dy.
\end{equation}

Now, we show that the moving spheres method can be started from infinity.
\begin{pro}\label{c39}
If $ \lam >0 $ is large enough, we have
\[
w_{\lam}(x) \geq 0 \quad \text{and} \quad z_{\lam}(x) \geq 0 \ \text{in} \ B_{\lam}(0)\setminus \{0\}.
\]
\end{pro}
\begin{proof} We denote
\[
A_1 =\left\{x\in \mathbb{R}^n | R_0 \le |x| \le \frac{\lam}{2}\right\}, \quad A_2 =\left\{x\in \mathbb{R}^n | \frac{\lam}{2} \le |x| \le \lam \right\}
\]
and
\[
A_3= B_{R_0}(0)\setminus \{0\}.
\]

First, if $ x\in A_1$, since $ u$ is continuous at $ 0$, we get
\[
w_\lam(x)= u\left(\frac{x}{|x|^2}\right) - 2\ln |x| - u\left(\frac{x}{\lam^2}\right) + 2\ln \lam \ge 2\ln 2 - o(1)\ge 0
\]
for $ \lam$ and $ R_0$ large enough. Similarly, for $ \lam$ large enough,
\begin{equation}
\begin{aligned}
z_{\lam}(x) & = \frac{1}{|x|^{\mu}}v\left( \frac{x}{|x|^2}\right) -\frac{1}{\lam^{\mu}}v\left( \frac{x}{\lam^2}\right) \\
& = \left(\frac{1}{|x|^\mu} - \frac{1}{\lam^\mu}\right) v\left( \frac{x}{|x|^2}\right) + \frac{1}{\lam^\mu} \left[ v\left( \frac{x}{|x|^2}\right) -v\left( \frac{x}{\lam^2}\right) \right] \\
& \ge \frac{2^\mu -1}{\lam^\mu} (v(0)+ o(1)) - \frac{C}{\lam^\mu} \left(\frac{1}{R_0} - \frac{R_0}{\lam^2}\right) \\
& \ge 0,
\end{aligned}
\end{equation}
where $ C= \max\limits_{x\in B_1} |\nabla v(x)|$ and $ v(0) > 0$.

Next, we consider that $ x\in A_2$. We deduce from $ v(0) > 0$ that
\[
\nabla (|x|^{\frac{\mu}{2}}v(x)) \cdot x = \frac{\mu}{2} |x|^{\frac{\mu}{2}}v(x) + |x|^{\frac{\mu}{2}} \nabla v(x) \cdot x >0
\]
for $ |x|$ small enough. The function $ |x|^{\frac{\mu}{2}}v(x)$ is increasing in the direction of $ x$. Hence,
\[
\left|\frac{x}{|x|^2} \right|^{\frac{\mu}{2}} v\left( \frac{x}{|x|^2}\right) \ge \left|\frac{x}{\lam^2} \right|^{\frac{\mu}{2}} v\left( \frac{x}{\lam^2}\right)
\]
and
\[
q(x)\ge q_{\lam}(x), \quad x\in A_2.
\]
For $ \frac{\lam}{2} \le |x| \le \lam$, a directly calculation shows that
\begin{equation}
\begin{aligned}
- \Delta w_\lam(x) =&  - \Delta_x u\left(\frac{x}{|x|^2}\right) + \frac{2(n-2)}{|x|^2} + \Delta_x u\left(\frac{x}{\lam^2}\right) \\
=& \frac{1}{|x|^2} \left\{ 2(n-2) + 2(n-2) [y \cdot \nabla u(y)]_{y= x /|x|^2} \right\} \\
& +  \frac{1}{|x|^2} \left\{ \frac{|x|^2}{\lam^4} (\Delta u)\left(\frac{x}{\lam^2}\right)- \frac{1}{|x|^2}(\Delta u)\left(\frac{x}{|x|^2}\right) \right\} \\
\ge & 0
\end{aligned}
\end{equation}
for $ \lam$ large enough. From the maximum principle we conclude that
\[
w_\lam(x) \ge 0 \quad \text{in} \quad A_2.
\]

Now, we consider the case $ x \in A_3$. By \eqref{c70}, one has
\[
w_\lam(x)= u\left(\frac{x}{|x|^2}\right) - 2\ln |x| - u\left(\frac{x}{\lam^2}\right) + 2\ln \lam \ge -C +2\ln \lam
\]
for $ |x|$ small enough. Choosing $ \lam $ large enough, one gets
\[
w_\lam(x) \ge 0 \quad \text{in} \quad B_{R_0}\setminus \{0\}.
\]
Since $ |x-y| \le \left|x-\frac{\lam^2 y}{|y|^2}\right|\frac{|y|}{\lam} $, then
\[
\frac{1}{|x-y|^\mu} - \frac{1}{\left|x-\frac{\lam^2 y}{|y|^2}\right|^\mu}\frac{\lam^\mu}{|y|^\mu} \geq 0
\]
for $ x\in B_{R_0}\setminus \{0\}$, $ y\in B_{\lam}\setminus \{0\}$ and $ w_\lam(x) \ge 0$ in $ B_{\lam}\setminus \{0\}$.
We infer from \eqref{c38} that
\[
z_\lam(x) \ge 0 \quad \text{in} \quad B_{R_0}\setminus \{0\}.
\]
This completes the proof of this proposition.
\end{proof}

For fixed $ b\in \mathbb{R}^n$, we define
\begin{equation*}
\begin{split}
&  u_b(x) = u(x+b),\quad v_b(x) = v(x+b)\\
& p_b(x) = u_b\left(\frac{x}{|x|^2}\right) - 2\ln |x|, \quad
q_b(x) = \frac{1}{|x|^{\mu}} v_b\left( \frac{x}{|x|^2}\right),\\
& p_{\lam,b}(x) = p_b\left(\frac{\lam^2 x}{|x|^2}\right) +2\ln \frac{\lam}{|x|} =u_b\left(\frac{x}{\lam^2}\right) - 2\ln \lam, \\
& q_{\lam,b}(x) = \frac{\lam^\mu}{|x|^{\mu}}q_b\left(\frac{\lam^2 x}{|x|^2}\right) = \frac{1}{\lam^{\mu}}v_b\left( \frac{x}{\lam^2}\right), \\
& w_{\lam,b}(x) = p_b(x) - p_{\lam,b}(x), \quad z_{\lam,b}(x) = q_b(x) - q_{\lam,b}(x).
\end{split}
\end{equation*}
Next, we give a technical lemma.

\begin{lemma}\label{c40}(see Lemma 11.2 of \cite{LiyyZhang} or see Lemma 3.3 of \cite{Yu})
(1)Suppose $ v \in C^1(\mathbb{R}^n)$, if for all $ b\in \mathbb{R}^n$ and $ \lam > 0$, the following inequality holds
\[
\frac{1}{|x|^\mu} v_b \left(\frac{x}{|x|^2}\right) - \frac{1}{\lam^\mu} v_b \left(\frac{x}{\lam^2}\right) \geq 0, \ \forall x \in B_\lam \setminus \{0\},
\]
then we have $ v(x) \equiv C$.

(2) Suppose $ u \in C^1(\mathbb{R}^n)$, if for all $ b\in \mathbb{R}^n$ and $ \lam > 0$, the following inequality holds
\[
u_b \left(\frac{x}{|x|^2}\right)-2 \ln |x| - u_b \left(\frac{x}{\lam^2}\right) + 2 \ln \lam \geq 0, \quad \forall x \in B_\lam \setminus \{0\}
\]
then we have $ u(x) \equiv C$.
\end{lemma}

For fixed $ b\in \mathbb{R}^n$, we define
\[
\lam_b = \inf\{\lam > 0 \ | \ w_{\mu,b}(x)\ge 0,  z_{\mu,b}(x) \ge 0 \ \text{in} \ B_\mu \setminus \{0\}, \lam \leq \mu < \infty\}.
\]

\begin{pro}\label{c46}
There exists a vector $ \bar{b}\in \mathbb{R}^n$, such that $ \lam_{\bar{b}} > 0$.
\end{pro}
\begin{proof} We prove it by contradiction. Suppose on the contrary, then for any $ b\in \mathbb{R}^n$, we have $ \lam_b =0$. By the definition of $ \lam_b$, we get
\[
w_{\lam,b}(x)\ge 0 \quad \text{and} \quad   z_{\lam,b}(x) \ge 0
\]
for any $ \lam > 0$ and $ x\in B_\lam \setminus \{0\}$. Then we infer from Lemma \ref{c40} that $ u(x)\equiv C_1$ and $ v(x)\equiv C_2$. This contradicts $ \int_{\mathbb{R}^n}e^{\frac{2n- \mu}{2}u(y)} dy < \infty$.
\end{proof}

\begin{pro}\label{c47}
Suppose that $ \lam_b >0$ for some $ b\in \mathbb{R}^n$, then we have
\[
w_{\lam_b,b}(x)\equiv 0 \quad \text{and} \quad   z_{\lam_b,b}(x) \equiv 0 \quad \forall x\in B_{\lam_b} \setminus \{0\}.
\]
\end{pro}
\begin{proof}
Without loss of generality, we assume $b=0$. Then we define
\[
w_{\lam_0} = w_{\lam_0,0}\quad \text{and} \quad  z_{\lam_0} = z_{\lam_0,0}.
\]

Suppose on the contrary that $ w_{\lam_0}(x) \not\equiv 0$ or $ z_{\lam_0}(x) \not\equiv 0$. By the definition of $ \lam_0$ we deduce that
\begin{equation}\label{c43}
w_{\lam_0}(x) \geq 0 \quad \text{and} \quad z_{\lam_0}(x)\geq 0 \quad \text{in} \quad B_{\lam_0} \setminus \{0\}.
\end{equation}
By a direct calculation, one gets
\begin{equation}\label{c41}
p_\lam(x) = \frac{1}{\beta_n}  \int_{\mathbb{R}^n}\left[ \ln \left(\frac{|y|+\lam^2}{|x-y|}\right)\right] q_\lam(y) e^{\frac{2n-\mu}{2}p_\lam(y)} dy + C - 2\ln \lam.
\end{equation}
Moreover, we deduce from \eqref{c42} and \eqref{c41} that
\[
-\Delta w_{\lam}(x) =\frac{n-2}{\beta_n} \int_{B_{\lam}} \left(\frac{1}{|x-y|^2} - \frac{1}{\left|x - \frac{\lam^2 y}{|y|^2}\right|^2}\right) (q(y) e^{\frac{2n-\mu}{2}p(y)} - q_\lam(y) e^{\frac{2n-\mu}{2}p_\lam(y)}) dy.
\]
By \eqref{c43}, we conclude that
\[
-\Delta w_{\lam_0}(x) \ge 0 \quad \text{in} \quad B_{\lam_0} \setminus \{0\}.
\]
We infer from the maximum principle that
\begin{equation}\label{c44}
w_{\lam_0}(x) > 0\quad \text{in} \quad B_{\lam_0} \setminus \{0\}.
\end{equation}
Substituting the inequality in \eqref{c44} into \eqref{c38}, we obtain
\[
z_{\lam_0}(x) > 0\quad \text{in} \quad B_{\lam_0} \setminus \{0\}.
\]
By the definition of $ \lam_0$, there exists a sequence $ \lam_k < \lam_0$, $ \lam_k \rightarrow \lam_0$, such that
\[
\inf\limits_{x\in B_{\lam_k}\setminus \{0\}} w_{\lam_k}(x) < 0.
\]
We deduce from \eqref{c44} and the Hopf Lemma that
\be\label{c45}
\frac{\ptl w_{\lam_0}}{\ptl \nu} (x) < 0 \quad \text{on} \quad \ptl B_{\lam_0},
\ee
where $ \nu$ is the unit outer normal direction.
\par Claim:  there exists a $ \gamma= \gamma(\lam_0) >0$, such that
\[
w_{\lam_k}(x) \geq \frac{\gamma}{2}, \quad   \forall x\in B_{\frac{\lam_0}{2}} \setminus \{0\}.
\]
Let
\[
\gamma = \min_{\ptl B_{\frac{\lam_0}{2}}} w_{\lam_0}(x) > 0.
\]
We define
\[
h(x) = \gamma - \frac{r^{n-2}}{|x|^{n-2}}\gamma \quad \text{in} \quad B_{\frac{\lam_0}{2}} \setminus B_r
\]
with $ r$ small. Then, $ k(x) = w_{\lam_0}(x) - h(x)$ satisfies
\begin{equation}
\left\{
\begin{aligned}
& - \Delta k(x) = - \Delta w_{\lam_0}(x) \geq 0 \quad \text{in} \quad B_{\frac{\lam_0}{2}} \setminus B_r, \\
& k(x) = w_{\lam_0}(x) > 0 \qquad \text{on} \quad \ptl B_r \\
& k(x)> 0 \qquad \qquad \text{on} \quad \ptl B_{\frac{\lam_0}{2}}.
\nonumber
\end{aligned}
\right.
\end{equation}
Hence, by the maximum principle and letting $ r\rightarrow 0$, one gets $ w_{\lam_0}(x) \ge \gamma$, in $B_{\frac{\lambda_0}{2}}\backslash\{0\}$.  Then
\begin{equation*}
\begin{split}
w_{\lam_k}(x)=&p(x)-p_{\lambda_k}(x)\\
=&w_{\lam_0}(x)+ p_{\lambda_0}(x)-p_{\lambda_k}(x)\\
\ge & \frac{\gamma}{2}
\end{split}
\end{equation*}
provided $\lambda_k$ is close enough to $\lambda_0$. This proves the claim.

On the other hand, since $ \inf\limits_{B_{\lam_k}\setminus \{0\}} w_{\lam_k}(x) < 0$, then we infer from the claim that there exists an $ x_k \in B_{\lam_k}\setminus B_{\frac{\lam_0}{2}} $ such that
\[
w_{\lam_k}(x_k) = \inf_{B_{\lam_k}\setminus \{0\}} w_{\lam_k}(x) < 0.
\]
In particular, we have $ \nabla w_{\lam_k}(x_k) = 0$. We assume that, up to a subsequence, $ x_k \rightarrow \bar{x}$, then we obtain $ \nabla w_{\lam_0}(\bar{x}) =0 $ and $ w_{\lam_0}(\bar{x}) =0$. Hence $ \bar{x} \in \ptl B_{\lam_0}$. However, this contradicts \eqref{c45}.

Therefore, for any sequence $ \lam_k < \lam_0$, $ \lam_k \rightarrow \lam_0$, it is clear that
\[
w_{\lam_k}(x) \ge 0, \quad \text{in} \quad B_{\lam_k}\setminus \{0\}.
\]
It follows from \eqref{c38} that
\[
z_{\lam_k}(x) \ge 0, \quad \text{in} \quad B_{\lam_k}\setminus \{0\}.
\]
This contradicts the definition of $ \lam_0$.
\end{proof}

\begin{pro}\label{c50}
For any $ b \in \mathbb{R}^n$, we have $ \lam_b >0$.
\end{pro}
\begin{proof} By Proposition \ref{c46} and Proposition \ref{c47}, there exists a vector $ \bar{b}\in \mathbb{R}^n$, such that $ \lam_{\bar{b}} > 0$,
\[
w_{\lam_{\bar{b}},{\bar{b}}}(x)\equiv 0 \quad \text{and} \quad   z_{\lam_{\bar{b}},{\bar{b}}}(x) \equiv 0 \quad x\in B_{\lam_{\bar{b}}} \setminus \{0\}.
\]
That is
\[
u_{\bar{b}} \left(\frac{x}{|x|^2}\right)-2 \ln |x| - u_{\bar{b}} \left(\frac{x}{\lam_{\bar{b}}^2}\right) + 2 \ln \lam_{\bar{b}} \equiv 0, \quad x \in B_{\lam_{\bar{b}} } \setminus \{0\}.
\]
Letting $ |x| \rightarrow 0$, we obtain
\be\label{c48}
\lim_{|x| \rightarrow 0} \left( u_{\bar{b}} \left(\frac{x}{|x|^2}\right)-2 \ln |x| \right) = u_{\bar{b}}(0) - 2 \ln \lam_{\bar{b}}.
\ee

Suppose on the contrary that there exists a vector $ b \in \mathbb{R}^n$ such that $ \lam_b = 0$, then we obtain
\[
u_b \left(\frac{x}{|x|^2}\right)-2 \ln |x| - u_b \left(\frac{x}{\lam^2}\right) + 2 \ln \lam \geq 0
\]
for all $ \lam > 0$ and $ x \in B_\lam \setminus \{0\}$. Fix $ \lam$ and $ |x| \rightarrow 0$, we infer
\be\label{c49}
\liminf_{|x| \rightarrow 0} \left(u_b \left(\frac{x}{|x|^2}\right)-2 \ln |x| \right) \ge u_b(0) -2\ln \lam.
\ee
We derive from \eqref{c48} and \eqref{c49} that
\[
u_{\bar{b}}(0) - 2 \ln \lam_{\bar{b}}  \ge u_b(0) -2\ln \lam,
\]
this is a contradiction for small $ \lam$.
\end{proof}

\begin{pro}\label{c51}
For all $ b \in \mathbb{R}^n $, we have $ \lam_b > 0$, $ w_{\lam_b,b} \equiv 0$ and $ z_{\lam_b,b} \equiv 0$ in $ B_{\lam_b}\setminus \{0\}$.
\end{pro}
\begin{proof} This is a direct consequence of Proposition \ref{c47}
and Proposition \ref{c50}.
\end{proof}

\begin{proof}[Proof of Theorem \ref{c4}]
Let
\[
f(x) = e^{u(x)}.
\]
Then it follows from Proposition \ref{c51} that
\[
f(x) = \frac{1}{\lam_b^2 |x-b|^2} f \left(\frac{x-b}{\lam_b^2|x-b|^2}+b\right).
\]
Let
\[
A:= \lim_{|x|\rightarrow \infty} |x|^2 f(x) = \frac{f(b)}{\lam_b^2} = \frac{f(0)}{\lam_0^2},
\]
where $ A > 0$. We first assume that $ A=1$. Since
\[
f(x) =\frac{1}{\lam_0^2|x|^2} f \left(\frac{ x}{\lam_0^2|x|^2}\right) = \frac{1}{\lam_b^2 |x-b|^2} f \left(\frac{x-b}{\lam_b^2|x-b|^2}+b\right),
\]
then one has
\be\label{c52}
f(x) =\frac{1}{\lam_0^2|x|^2} \left[f(0)+ \frac{ x\cdot \nabla f(0)}{\lam_0^2|x|^2} + o\left(\frac{1}{|x|}\right)\right]
\ee
and
\be\label{c53}
f(x) =\frac{1}{\lam_b^2 |x-b|^2} \left[f(b)+  \frac{(x-b)\cdot \nabla f(b)}{\lam_b^2 |x-b|^2} + o\left(\frac{1}{|x-b|}\right)\right]
\ee
as $ |x|\rightarrow \infty$. It follows from equation \eqref{c52} and \eqref{c53} that
\[
\frac{\ptl f(b)}{\ptl x_i} f(b)^{-2} = \frac{\ptl f(0)}{\ptl x_i} f(0)^{-2} - 2b_i
\]
and
\[
(f^{-1})_i(b) = 2b_i + (f^{-1})_i(0) = \frac{\ptl}{\ptl b_i}( |b|^2 + \nabla f^{-1}(0) \cdot b).
\]
Therefore,
\[
f(b) = \frac{1}{|b-d_0|^2 + l}
\]
and
\[
u(b) = \ln \left( \frac{1}{|b-d_0|^2 + l} \right),
\]
where $ l$ is a constant. Finally, if we don't assume $ A =1$, then
\[
u(x)= \ln  \frac{C_1(\eps)}{|x-x_0|^2 + \eps^2}.
\]
This completes the proof of Theorem \ref{c4}.
\end{proof}

\section{Acknowledgments.}
This work is partially supported by National Natural Science Foundation of China 12141105.

\bibliographystyle{plain}
\bibliography{Choquard}

\end{document}